\documentclass[11pt,reqno]{amsart}
\usepackage{amsmath}
\usepackage[T1]{fontenc}
\usepackage{graphicx}
\usepackage{cite}
\usepackage[colorlinks=true,linkcolor=MyLinkColor,citecolor=MyLinkColor]{hyperref} 
\usepackage{cleveref}
\usepackage{color}
\definecolor{MyLinkColor}{rgb}{0,0,0.4}
\newtheorem{theorem}{Theorem}[section]

\newtheorem{lemma}[theorem]{Lemma}
\newtheorem{proposition}[theorem]{Proposition}

\newtheorem{remark}[theorem]{Remark}
\numberwithin{equation}{section}
\allowdisplaybreaks
\begin{document}
\title[]{Bounded weak solutions to a class of\\ degenerate cross-diffusion systems}
\thanks{}
\author{Philippe Lauren\c{c}ot}
\address{Institut de Math\'ematiques de Toulouse, UMR~5219, Universit\'e de Toulouse, CNRS \\ F--31062 Toulouse Cedex 9, France}
\email{laurenco@math.univ-toulouse.fr}
\author{Bogdan-Vasile Matioc}
\address{Fakult\"at f\"ur Mathematik, Universit\"at Regensburg \\ D--93040 Regensburg, Deutschland}
\email{bogdan.matioc@ur.de}


\date{\today}

%%%%%%%%%%%%%%%%
%%%%%%%%%%%%%%%%
\begin{abstract}
Bounded weak solutions are constructed for a degenerate parabolic system with a full diffusion matrix, which is a generalized version of the thin film Muskat system. Boundedness is achieved with the help of a sequence $(\mathcal{E}_n)_{n\ge 2}$ of Liapunov functionals such that $\mathcal{E}_n$ is equivalent to the $L_n$-norm for each $n\ge 2$ and $\mathcal{E}_n^{1/n}$ controls the $L_\infty$-norm in the limit $n\to\infty$. Weak solutions are built by a compactness approach, special care being needed in the construction of the approximation in order to preserve the availability of the above-mentioned Liapunov functionals.
\end{abstract}
%%%%%%%%%%%%%%%%
%%%%%%%%%%%%%%%%

\maketitle

%
%     HEADLINES
%
\pagestyle{myheadings}
\markboth{\sc{Ph.~Lauren\c cot \& B.-V.~Matioc}}{\sc{Bounded weak solutions to cross-diffusion systems}}

%%%%%%%%%%%%%%%%
%%%%%%%%%%%%%%%%
\section{Introduction}\label{sec1}
%%%%%%%%%%%%%%%%
%%%%%%%%%%%%%%%%

Let $\Omega$ be a bounded domain of $\mathbb{R}^N$, $N\ge 1$, with smooth boundary $\partial\Omega$ and let $R$ and $\mu$ be two positive real numbers. In a recent paper~\cite{LM2021b}, we noticed that there is an infinite family $(\mathcal{E}_n)_{n\ge 1}$ of Liapunov functionals associated with the thin film Muskat system
\begin{align*}
	\partial_t f & = \mathrm{div}\left(f \nabla\left[ (1+R)f + R g \right] \right) \;\text{ in }\; (0,\infty)\times \Omega\,, \\
	\partial_t g & = \mu R\, \mathrm{div}\left(g \nabla\left[ f + g \right] \right) \;\text{ in }\; (0,\infty)\times \Omega\,, 
\end{align*}
supplemented with homogeneous Neumann boundary conditions and initial conditions, with the following properties: for all $n\ge 2$, there are $0<c_n < C_n$ such that 
\begin{equation*}
c_n \|f+g\|_n^n \le \mathcal{E}_n(f,g) \le C_n \|f+g\|_n ^n\,, \qquad (f,g)\in L_{n,+}(\Omega,\mathbb{R}^2)\,,
\end{equation*}
and there are $0<c_\infty < C_\infty$ such that
\begin{equation*}
	c_\infty \|f+g\|_\infty \le \liminf_{n\to\infty} \mathcal{E}_n(f,g)^{1/n} \le \limsup_{n\to\infty} \mathcal{E}_n(f,g)^{1/n} \le C_\infty \|f+g\|_\infty
\end{equation*}
for $(f,g)\in L_{\infty,+}(\Omega,\mathbb{R}^2)$, 
where $L_{p,+}(\Omega,\mathbb{R}^m)$ denotes the positive cone of $L_p(\Omega,\mathbb{R}^m)$ for $m\ge 1$ and~${p\in [1,\infty]}$. 
On the one hand, the thin film Muskat system being of cross-diffusion type (i.e., featuring a diffusion matrix with no zero entry), 
the availability of such a family of Liapunov functionals is rather seldom within this class of systems and paves the way towards the construction of bounded weak solutions, a result that we were only able to show in one space dimension $N=1$ in \cite{LM2021b}.
 On the other hand, it is tempting to figure out whether this property is peculiar to the thin film Muskat system or extends to the generalization thereof
\begin{subequations}\label{PB}
	\begin{align}
		\partial_t f & = \mathrm{div}\left(f \nabla\left[ af + b g \right] \right) \;\text{ in }\; (0,\infty)\times \Omega\,, \label{PB1a} \\
		\partial_t g & = \mathrm{div}\left(g \nabla\left[ cf + d g \right] \right) \;\text{ in }\; (0,\infty)\times \Omega\,, \label{PB1b}
	\end{align}
with $(a,\, b,\, c,\, d) \in (0,\infty)^4$, supplemented with homogeneous Neumann boundary conditions
\begin{equation}
	\nabla f\cdot \mathbf{n} = \nabla g\cdot \mathbf{n} = 0 \;\text{ on }\; (0,\infty)\times \partial\Omega\,, \label{PB1c}
\end{equation}
and non-negative initial conditions
\begin{equation}
	(f,g)(0) = (f^{in},g^{in}) \;\text{ in }\; \Omega\,. \label{PB1d}
\end{equation}
\end{subequations}
Obviously, the thin film Muskat system is a particular case of~\eqref{PB1a}-\eqref{PB1b}, corresponding to the choice~${(a,\, b,\, c,\, d)=(1+R, \, R,\, \mu R,\, \mu R)}$. 

 The main result of this paper is to show that, for any quadruple $(a,\, b,\, c,\, d)$ satisfying
 \begin{equation}\label{condabcd}
 	(a,\, b,\, c,\, d) \in(0,\infty)^4 \qquad\text{and}\qquad ad>bc\,,
 \end{equation}
we can associate a similar family of Liapunov functionals with~\eqref{PB} and prove the global existence of bounded non-negative weak solutions to~\eqref{PB}, whatever the dimension $N\ge 1$.
 More precisely, given a quadruple $(a,\, b,\, c,\, d)$ satisfying~\eqref{condabcd},
  we define a sequence $(\Phi_n)_{n\ge 1}$ of functions as follows. Setting $L(r):=r\ln{r}-r+1\ge 0$, $r\ge 0$, we first define the function $\Phi_1$ by the relation  
\begin{equation}
\Phi_1(X) := L(X_1) + \cfrac{b^2}{ad}L(X_2)\,, \qquad X=(X_1,X_2)\in [0,\infty)^2\,. \label{p4d}
\end{equation}
Next, for each integer $n\geq 2$, let $\Phi_n$ be the homogeneous polynomial of degree $n$ defined by
\begin{equation}
	\Phi_n(X) := \sum_{j=0}^n a_{j,n} X_1^j X_2^{n-j}\,, \qquad X=(X_1,X_2)\in \mathbb{R}^2\,, \label{p4b}
\end{equation}
with $a_{0,n}:=1$ and 
\begin{equation} \label{p4c}
	a_{j,n}:=\binom{n}{j} \prod_{k=0}^{j-1}\frac{ak+c(n-k-1)}{bk+d(n-k-1)}> 0\,,\qquad 1\leq j\leq n\,.
\end{equation}
We then define, for $n\ge 1$, the functional 
\begin{equation}\label{p4}
	\mathcal{E}_n(u):=\int_{\Omega}\Phi_n(u(x))\, dx,\qquad u=(f,g)\in L_{\max\{2,n\},+}(\Omega,\mathbb{R}^2)\,. 
\end{equation}
We finally  observe that \eqref{condabcd} guarantees that
\begin{equation}\label{constants}
	\Theta_1:=\frac{b(ad+bc)}{2ad}>0\qquad\text{and}\qquad \Theta_2:=\frac{(ad-bc)(3ad+bc)}{4a^2d^2}>0\,.
\end{equation}
With this notation, the main result of this paper is the following:

%%%%%%%%%%%%%%%%
\begin{theorem}\label{ThBWS}
	Assume \eqref{condabcd} and let $ u^{in} := (f^{in},g^{in})\in L_{\infty,+}(\Omega,\mathbb{R}^2)$ be given.
	Then, there is a bounded weak solution $u=(f,g)$ to \eqref{PB} such that:
	\begin{itemize}
		\item[(i)] for each $T>0$, 
		\begin{equation}
			(f,g)\in L_{\infty,+}((0,T)\times \Omega,\mathbb{R}^2) \cap L_2((0,T),H^1(\Omega,\mathbb{R}^2))\cap W_2^1((0,T),H^1 (\Omega,\mathbb{R}^2)')\,; \label{p1}
		\end{equation}
		\item[(ii)] for all $\varphi\in H^1(\Omega)$ and $t\ge 0$,
		\begin{subequations}\label{p2}
			\begin{equation}
				\int_\Omega (f(t,x)-f^{in}(x)) \varphi(x)\ \mathrm{d}x  +\int_0^t \int_\Omega f(s,x) \nabla[ a f + b g ](s,x) \cdot \nabla\varphi(x)\, \mathrm{d}x\mathrm{d}s =0 \label{p2a}
			\end{equation} 
			and
			\begin{equation}
				\int_\Omega (g(t,x)-g^{in}(x)) \varphi(x)\ \mathrm{d}x +  \int_0^t \int_\Omega g(s,x) \nabla[cf + dg ](s,x) \cdot \nabla\varphi(x)\, \mathrm{d}x\mathrm{d}s =0\,; \label{p2b}
			\end{equation}
		\end{subequations} 
		\item[(iii)] for all $t\ge 0$, 
		\begin{equation}\label{p3}
			\begin{aligned}
				&\mathcal{E}_1(u(t))+  \frac{1}{a}\int_0^t \int_\Omega\big[|\nabla (af+\Theta_1 g)|^2+\Theta_2 |\nabla g|^2\big](s,x)\ \mathrm{d}x\mathrm{d}s\le \mathcal{E}_1(u^{in})\,,
			\end{aligned}
		\end{equation}
		where the positive constants $\Theta_1$ and $\Theta_2$ are defined in \eqref{constants};
		\item[(iv)] for all $n\ge 2$ and all $t\ge 0$, 
		\begin{equation}
			\mathcal{E}_n(u(t)) \le \mathcal{E}_n(u^{in})\,; \label{p4a}
		\end{equation}
		\item[(v)] for $t\ge 0$,
		\begin{equation}
			\|f(t)+g(t)\|_\infty\le \frac{d}{b}\frac{\max\{a,\,b\}}{\min\{c,\, d\}} \|f^{in}+g^{in}\|_\infty\,.  \label{p5}
		\end{equation} 
	\end{itemize}
\end{theorem} 
%%%%%%%%%%%%%%%%

Let us first mention that \Cref{ThBWS} improves \cite{LM2021b} in two directions: on the one hand, it shows that the structural properties~\eqref{p3}, \eqref{p4a}, and~\eqref{p5}, uncovered there for the thin film Muskat system, are also available for the whole class~\eqref{PB}. On the other hand, it provides the existence of non-negative bounded weak solutions to~\eqref{PB} in all space dimensions, a result which was only established in one space dimension in \cite{LM2021b}. Global weak solutions to the thin film Muskat system are also constructed in \cite{ACCL2019, AIJM2018, BGB2019, ELM2011, LM2013, LM2017}, but they need not be bounded, except in \cite{BGB2019}. The latter however requires some smallness condition on the initial data, in contrast to \Cref{ThBWS}. Finally, the local well-posedness of the thin film Muskat system in the classical sense is investigated in~\cite{EMM2012}.

\bigskip

We next outline the main steps of the proof of \Cref{ThBWS}. As in \cite{LM2021b}, the starting point is to notice that, introducing the mobility matrix
\begin{equation}
	M(X)  =  (m_{jk}(X))_{1\le j,k\le 2} := \begin{pmatrix}
		a X_1& b X_1 \\ 
		c X_2 & d X_2
	\end{pmatrix}\,, \qquad X=(X_1,X_2)\in \mathbb{R}^2\,, \label{x1}
\end{equation}
and $u:=(f,g)$, an alternative formulation of the system~\eqref{PB1a}-\eqref{PB1b} is
\begin{equation}
	\partial_t u = \sum_{i=1}^N\partial_i (M(u)\partial_i u) \;\text{ in }\; (0,\infty)\times \Omega\,. \label{x2}
\end{equation} 
Then, given $\Phi\in C^2(\mathbb{R}^2,\mathbb{R})$, it readily follows from \eqref{x2}, the homogeneous Neumann boundary conditions~\eqref{PB1c}, and the symmetry of the Hessian matrix $D^2(\Phi)$ that
\begin{equation}
	\frac{\mathrm{d}}{\mathrm{d}t} \int_\Omega \Phi(u)\ \mathrm{d}x + \sum_{i=1}^N\int_\Omega \langle D^2\Phi(u) M(u) \partial_i u , \partial_i u \rangle\ \mathrm{d}x = 0\,, \label{x3}
\end{equation}
where $\langle \cdot,\cdot \rangle$ stands for the scalar product on $\mathbb{R}^2$. 
As a straightforward  consequence of \eqref{x3} we note that $\int_\Omega \Phi(u)\ \mathrm{d}x$ is a Liapunov functional for~\eqref{x2} when the matrix $D^2\Phi(u) M(u)$ is positive semidefinite. We shall then show in \Cref{secA} that, for all $n\ge 2$, it is possible to construct an homogeneous polynomial $\Phi_n\in \mathbb{R}[X_1,X_2]$ of degree $n$ which is convex on $[0,\infty)^2$ and such that the matrix $D^2\Phi_n(X) M(X)$ is positive semidefinite for all $X\in [0,\infty)^2$. A closed form formula is actually available for the polynomial $\Phi_n$, see \eqref{p4b} and \eqref{p4c}.  

We next construct weak solutions to~\eqref{x2} by a compactness method. It is here of utmost importance to construct approximations 
which do not alter the inequalities~\eqref{x3} for $\Phi=\Phi_n$ and~${n\ge 1}$. 
As a first step, it is well-known that implicit time discrete schemes are well-suited in that direction. 
Thus, given $\tau>0$, we shall first prove the existence of  a sequence $(u^\tau_l)_{l \ge 0}$
 which satisfies~${u^\tau_0=u^{in}:=(f^{in},g^{in})}$ and, for $l\ge 0$,
\begin{equation}
	u^\tau_{l+1} - \tau \sum_{i=1}^N\partial_i \Big(M(u^\tau_{l+1})\partial_i u^\tau_{l+1}\Big) = u^\tau_l \;\;\text{ in }\;\; \Omega\,, \label{x4}
\end{equation}
supplemented with homogeneous Neumann boundary conditions. Furthermore, the sequence   $(u^\tau_l)_{l \ge 0}$ has the property that, for $n\ge 1$ and $l\ge 0$,
\begin{equation}
	\mathcal{E}_n(u_{l+1}^\tau) + \tau \sum_{i=1}^N\int_\Omega \langle D^2\Phi_n(u_{l+1}^\tau) M(u_{l+1}^\tau) \partial_i u_{l+1}^\tau , \partial_i u_{l+1}^\tau \rangle\ \mathrm{d}x \le \mathcal{E}_n(u_{l}^\tau)\,, \label{x5}
\end{equation}
so that the structural property~\eqref{x3} is indeed preserved by the time discrete scheme. The existence of a solution to~\eqref{x4} is achieved by a compactness method relying on an approximation of the matrix $M(\cdot)$ by bounded ones. This step is actually the more delicate one, as we have to construct matrices approximating $M(\cdot)$ which do not alter~\eqref{x5}. To this end, a two-parameter approximation procedure is required and it is detailed in \Cref{sec2.2}. The existence of a weak solution to~\eqref{x4} satisfying~\eqref{x5} is shown in \Cref{sec2.4}, building upon preliminary and intermediate results established in \Cref{sec2.1} and \Cref{sec2.3}.

%%%%%%%%%%%%%%%%
\begin{remark}\label{rem.W}
	A common feature of system~\eqref{PB} is that it has, at least formally, a gradient flow structure for the functional $\mathcal{E}_2$ with respect to the $2$-Wasserstein distance in the space $\mathcal{P}_2(\Omega,\mathbb{R}^2)$ of probability measures with finite second moments, as pointed out in \cite{ACCL2019, LM2013} for the thin film Muskat system. 
	In particular, there is a natural variational structure associated with~\eqref{PB} which is suitable to construct weak solutions.
	 However, the connection between this variational structure and the whole family $(\mathcal{E}_n)_{n\ge 2}$ of Liapunov functionals is yet unclear.
\end{remark}
%%%%%%%%%%%%%%%%

\bigskip

\paragraph{\textbf{Notation.}} For $p\in [1,\infty]$, we denote the $L_p$-norm in $L_p(\Omega)$ by $\|\cdot\|_p$ and set 
\begin{equation*}
	L_p(\Omega,\mathbb{R}^2) := L_p(\Omega)\times L_p(\Omega)\,, \quad H^1(\Omega,\mathbb{R}^2):= H^1(\Omega)\times H^1(\Omega)\,.
\end{equation*}
The positive cone of a Banach lattice $E$ is denoted by $E_+$. The space of $2\times 2$ real-valued matrices is denoted by ${\mathbf{M}_2(\mathbb{R})}$, while ${\mathbf{Sym}_2(\mathbb{R})}$ is the subset of ${\mathbf{M}_2(\mathbb{R})}$ consisting of symmetric matrices and~${\mathbf{SPD}_2(\mathbb{R})}$ is the set of symmetric and positive definite matrices in $\mathbf{M}_2(\mathbb{R})$.
Finally, we denote the positive part of a real number $r\in\mathbb{R}$ by~${r_+:=\max\{r,0\}}$ and $\langle \cdot,\cdot \rangle$ is the scalar product on $\mathbb{R}^2$.

%%%%%%%%%%%%%%%%
%%%%%%%%%%%%%%%%
\section{A time discrete scheme}\label{sec2}
%%%%%%%%%%%%%%%%
%%%%%%%%%%%%%%%%

In order to construct bounded non-negative global weak solutions to the evolution problem~\eqref{PB}, we employ a compactness approach, paying special attention to preserve as much as possible the structural properties~\eqref{p3}, \eqref{p4a}, and~\eqref{p5} in the design of the approximation. It turns out that implicit time discrete schemes are well-suited for that purpose and we thus establish in this section the existence of solutions to the implicit time discrete scheme associated with~\eqref{PB}, see~\eqref{ex1a}-\eqref{ex1b}. 

%%%%%%%%%%%%%%%%
\begin{proposition}\label{P:1}
	Given  $\tau>0$ and $U=(F,G)\in L_{\infty,+}(\Omega,\mathbb{R}^2)$, there is a solution  
\begin{equation*}
		u=(f,g) \in H^1(\Omega,\mathbb{R}^2)\cap L_{\infty,+}(\Omega,\mathbb{R}^2)
\end{equation*} 
to 
\begin{subequations}\label{ex1}
\begin{align}
\int_\Omega \big( f \varphi + \tau f \nabla\left[ a f + b g \right] \cdot\nabla\varphi \big)\ \mathrm{d}x & = \int_\Omega F \varphi\ \mathrm{d}x\,, \qquad \varphi\in H^1(\Omega)\,, \label{ex1a} \\		
\int_\Omega \big( g \psi + \tau   g \nabla\left[ c f +d g \right]\cdot \nabla\psi \big)\ \mathrm{d}x & = \int_\Omega G \psi\ \mathrm{d}x\,, \qquad \psi\in H^1(\Omega)\,, \label{ex1b}
\end{align}
\end{subequations}
which also satisfies
\begin{equation}
	\mathcal{E}_n(u)  \le 	\mathcal{E}_n(U)\qquad\text{for $n\ge 2$}\label{ex2}
\end{equation}
and
\begin{equation}
	\mathcal{E}_1(u) + \frac{\tau}{a} \int_\Omega \big[  |\nabla (af+\Theta_1 g)|^2 +  \Theta_2|\nabla g|^2 \big]\ \mathrm{d}x \le 	\mathcal{E}_1(U) \,, \label{ex2b}
\end{equation}
recalling that, see~\eqref{constants},
\begin{equation*}
	\Theta_1 = \frac{b(ad+bc)}{2ad}>0 \;\;\text{ and }\;\; \Theta_2 = \frac{(ad-bc)(3ad+bc)}{4a^2 d^2}>0\,.
\end{equation*}	
\end{proposition}
%%%%%%%%%%%%%%%%

As already mentioned, several steps are involved in the proof of \Cref{P:1}. 
We begin with the existence of bounded weak solutions to an auxiliary elliptic system which shares the same structure with~\eqref{ex1},
 but has bounded coefficients instead of linearly growing ones, see Section~\ref{sec2.1}. As a next step, we introduce in Section~\ref{sec2.2} the approximation to \eqref{ex1} which is derived from~\eqref{ex1} by replacing the matrix $M(\cdot)$ defined in~\eqref{x1} by a suitable invertible and bounded matrix $M_\varepsilon^\rho(\cdot)$ with~$(\varepsilon,\rho)\in (0,1)\times (1,\infty)$. We emphasize here once more that the matrix $M_\varepsilon^\rho(\cdot)$ is designed in such a way that the inequalities~\eqref{ex2} and~\eqref{ex2b} are not significantly altered. 
  Passing to the limit, first as $\rho\to\infty$, and then as $\varepsilon\to 0$, is then performed in \Cref{sec2.3} and \Cref{sec2.4}, respectively, this last step completing the proof of \Cref{P:1}. 

Throughout this section, $C$ and $(C_l)_{l\ge 0}$ denote various positive constants depending only on~$N$,~$\Omega$, and $(a,\, b,\, c,\, d)$. 
Dependence upon additional parameters will be indicated explicitly.

%%%%%%%%%%%%%%%%
%%%%%%%%%%%%%%%%
\subsection{An auxiliary elliptic system}\label{sec2.1}
%%%%%%%%%%%%%%%%
%%%%%%%%%%%%%%%%

Let $A=(a_{jk})_{1\le j,k\le 2}$ and~$B=(b_{jk})_{1\le j,k\le 2}$ be chosen such that ${A\in \mathbf{SPD}_2(\mathbb{R})}$, ${B\in BC(\mathbb{R}^2,\mathbf{M}_2(\mathbb{R}))}$, and  $AB(X)\in \mathbf{SPD}_2(\mathbb{R})$ for all $X\in \mathbb{R}^2$.
Moreover, we assume that there is $\delta_1>0$ such that
\begin{equation}
	\langle AB(X)\xi,\xi \rangle \ge \delta_1 |\xi|^2\,, \qquad (X,\xi)\in \mathbb{R}^2\times\mathbb{R}^2\,. \label{ap1}
\end{equation}
Since $A\in \mathbf{SPD}_2(\mathbb{R})$, there is also $\delta_2>0$ such that
\begin{equation}
	\langle A\xi,\xi \rangle \ge\delta_2 |\xi|^2\,, \qquad \xi\in\mathbb{R}^2\,. \label{ap2}
\end{equation}

%%%%%%%%%%%%%%%%
\begin{lemma}\label{lem.ap1}
	Given $\tau>0$ and $U=(U_1,U_2)\in L_2(\Omega,\mathbb{R}^2)$, there is $u=(u_1,u_2)\in H^1(\Omega,\mathbb{R}^2)$ which solves the nonlinear equation
	\begin{equation}
		\int_\Omega \left[ \langle u , v \rangle + \tau \sum_{i=1}^N\langle B(u) \partial_i u , \partial_i v \rangle \right]\ \mathrm{d}x = \int_\Omega \langle U , v \rangle\ \mathrm{d}x\,, \qquad v\in H^1(\Omega,\mathbb{R}^2)\,. \label{ap3}
	\end{equation}
	Additionally:
	\begin{itemize}
		\item[(i)] If 
		\begin{equation}
			\begin{split}
				b_{11}(X) \geq  b_{12}(X) & = 0\,, \qquad X\in (-\infty,0)\times \mathbb{R}\,, \\
				b_{22}(X) \geq b_{21}(X) & = 0\,, \qquad X\in \mathbb{R}\times (-\infty,0)\,,
			\end{split} \label{ap10}
		\end{equation}
		and if $U(x)\in [0,\infty)^2$ for a.a. $x\in\Omega$, then $u(x)\in [0,\infty)^2$ for a.a. $x\in\Omega$.
		\item[(ii)] If there exists $\rho>0$ such that 
		\begin{equation}
			\begin{split}
				b_{11}(X) \geq  b_{12}(X) & = 0\,, \qquad X\in (\rho,\infty)\times \mathbb{R}\,, \\
				b_{22}(X) \geq b_{21}(X) & = 0\,, \qquad X\in \mathbb{R}\times (\rho,\infty)\,,
			\end{split} \label{ap10'}
		\end{equation}
		and if $\max\{U_1, U_2\} \leq \rho$ a.e. in $\Omega$, then $\max\{u_1, u_2\}\leq \rho$ a.e. in $\Omega$.
	\end{itemize}
	
\end{lemma}
%%%%%%%%%%%%%%%%

\begin{proof} 
The proof of \Cref{lem.ap1} is rather classical and it is actually similar to that of \cite[Lemma~B.1]{LM2021b}. We nevertheless sketch it below for the sake of completeness.

\medskip

\noindent\textbf{Step~1.} To set up a fixed point scheme, we consider $u\in L_2(\Omega,\mathbb{R}^2)$ and define a bilinear form $b_u$ on~${H^1(\Omega,\mathbb{R}^2)}$ by
	\begin{equation*}
		b_u(v,w) := \int_\Omega \left[ \langle Av , w \rangle + \tau \sum_{i=1}^N\langle AB(u) \partial_i v , \partial_i w \rangle \right]\ \mathrm{d}x\,, \qquad (v,w)\in H^1(\Omega,\mathbb{R}^2)\times H^1(\Omega,\mathbb{R}^2) \,.
	\end{equation*} 
	Owing to \eqref{ap1} and \eqref{ap2},
	\begin{equation}
		b_u(v,v) \ge \delta_0 \|v\|_{H^1}^2\,, \qquad v\in H^1(\Omega,\mathbb{R}^2)\,, \label{ap4}
	\end{equation}
	where $\delta_0:=\min\{\tau \delta_1, \delta_2\},$ while the boundedness of $B$ guarantees that
	\begin{equation*}
		|b_u(v,w)| \le b_* \|v\|_{H^1} \|w\|_{H^1}\,, \qquad (v,w)\in H^1(\Omega,\mathbb{R}^2)\times H^1(\Omega,\mathbb{R}^2) \,,
	\end{equation*}
	with
	\begin{equation*}
		b_* := 2\max_{1\le j,k\le 2}\{|a_{jk}|\} \left( 1 + 2\tau \max_{1\le j,k\le 2}\{\|b_{jk}\|_\infty\} \right)\,.
	\end{equation*}
	We then infer from Lax-Milgram's theorem that there is a unique $\mathcal{V}[u]\in H^1(\Omega,\mathbb{R}^2)$ such that
	\begin{equation}
		b_u(\mathcal{V}[u],w) = \int_\Omega \langle AU , w \rangle\ \mathrm{d}x\,, \qquad w\in H^1(\Omega,\mathbb{R}^2)\,. \label{ap5}
	\end{equation}
	An immediate consequence of  \eqref{ap4}, \eqref{ap5} (with $w=\mathcal{V}[u]$), and H\"older's inequality is the following estimate:
	\begin{equation*}
		\delta_0 \|\mathcal{V}[u]\|_{H^1}^ 2 \le b_u(\mathcal{V}[u],\mathcal{V}[u]) \le \|AU\|_2 \|\mathcal{V}[u]\|_2 \le \|AU\|_2 \|\mathcal{V}[u]\|_{H^1}\,. 
	\end{equation*}
	Hence
	\begin{equation}
		\|\mathcal{V}[u]\|_{H^1} \le \frac{\|AU\|_2}{\delta_0}\,. \label{ap6}
	\end{equation}
	
	We next argue as in the proof of \cite[Lemma~B.1]{LM2021b} to show that the map $\mathcal{V}$ is continuous and compact from $L_2(\Omega,\mathbb{R}^2)$ to itself, the proof relying on \eqref{ap6},
	 the compactness of the embedding of~$H^1(\Omega,\mathbb{R}^2)$ in $L_2(\Omega,\mathbb{R}^2)$, and the continuity and boundedness of $B$. 
	
	Consider now $\theta\in [0,1]$ and a function $u\in L_2(\Omega,\mathbb{R}^2)$ satisfying $u = \theta\mathcal{V}[u]$. Then $u\in H^1(\Omega,\mathbb{R}^2)$ and, in view of \eqref{ap6}, 
	\begin{equation*}
		\|u\|_2 = \theta \|\mathcal{V}[u]\|_2 \le \|\mathcal{V}[u]\|_2 \le \|\mathcal{V}[u]\|_{H^1} \le \frac{\|AU\|_2}{\delta_0}\,. 
	\end{equation*}
	Thanks to the above bound and the continuity and compactness properties of the map $\mathcal{V}$ in~$L_2(\Omega,\mathbb{R}^2)$, 
	we are in a position to apply Leray-Schauder's fixed point theorem, see \cite[Theorem~11.3]{GT2001} for instance, and conclude that the map $\mathcal{V}$ has a fixed point $u\in L_2(\Omega,\mathbb{R}^2)$. Since $\mathcal{V}$ ranges in $H^1(\Omega,\mathbb{R}^2)$, the function~$u$ actually belongs to $H^1(\Omega,\mathbb{R}^2)$ and satisfies 
	\begin{equation*}
		b_u(u,w) = \int_\Omega \langle AU , w \rangle\ \mathrm{d}x\,, \qquad w\in H^1(\Omega,\mathbb{R}^2)\,.
	\end{equation*}
	Finally, given $v\in H^1(\Omega,\mathbb{R}^2)$, the function $w=A^{-1}v$ also belongs to $H^1(\Omega,\mathbb{R}^2)$ and we infer from
	the above identity and the symmetry of $A$ that
	\begin{align*}
		\int_\Omega \langle U , v \rangle\ \mathrm{d}x  = \int_\Omega \langle AU , w \rangle\ \mathrm{d}x & = b_u(u,w) =b_u(u,A^{-1}v) \\
		  & = \int_\Omega \Big[ \langle u , v \rangle + \tau \sum_{i=1}^N\langle B(u) \partial_i u , \partial_i v \rangle \Big]\ \mathrm{d}x \,.
	\end{align*}
We have thus constructed a weak solution $u\in H^1(\Omega,\mathbb{R}^2)$ to \eqref{ap3}.
	
\medskip
	
\noindent\textbf{Step~2.} We now turn to the sign-preserving property~(i) and assume that $U(x)\in [0,\infty)^2$ for a.a.~${x\in\Omega}$. Let $u\in H^1(\Omega,\mathbb{R}^2)$ be a weak solution to \eqref{ap3} and set~$\varphi:=-u$. Then $(\varphi_{1,+},\varphi_{2,+})$ belongs to $H^1(\Omega,\mathbb{R}^2)$ and it follows from \eqref{ap3} that
\begin{equation}
	\begin{split}
		& \int_\Omega \Big[ \varphi_1 \varphi_{1,+}  + \varphi_2 \varphi_{2,+} + \tau \sum_{i=1}^N\sum_{j,k=1}^2 b_{jk}(u) \partial_i \varphi_k \partial_i (\varphi_{j,+}) \Big]\ \mathrm{d}x \\
		& \hspace{5cm} = - \int_\Omega \left( U_1 \varphi_{1,+}  + U_2 \varphi_{2,+} \right)\ \mathrm{d}x \le 0\,. 
	\end{split} \label{ap11}
\end{equation}
	We now infer from \eqref{ap10} that, for $1\le i \le N$, 
	\begin{align*}
		b_{11}(u) \partial_i\varphi_1 \partial_i\varphi_{1,+} & = b_{11}(u) \mathbf{1}_{(-\infty,0)}(u_1) |\partial_i u_1|^2\geq  0\,, \\
		b_{12}(u) \partial_i\varphi_2 \partial_i\varphi_{1,+} & = b_{12}(u) \mathbf{1}_{(-\infty,0)}(u_1) \partial_i u_1 \partial_i u_2 = 0\,, \\
		b_{21}(u) \partial_i\varphi_1 \partial_i\varphi_{2,+} & = b_{21}(u) \mathbf{1}_{(-\infty,0)}(u_2) \partial_i u_1 \partial_i u_2 = 0\,, \\
		b_{22}(u) \partial_i\varphi_2 \partial_i\varphi_{2,+} & = b_{22}(u) \mathbf{1}_{(-\infty,0)}(u_2) |\partial_i u_2|^2 \geq0\,,
	\end{align*}
	so that the second term on the left-hand side of \eqref{ap11} is non-negative. Consequently, \eqref{ap11} gives
	\begin{equation*}
		\int_\Omega \big[ |\varphi_{1,+}|^2  + |\varphi_{2,+}|^2 \big]\ \mathrm{d}x \le 0\,, 
	\end{equation*}
	which implies that $\varphi_{1,+}=\varphi_{2,+}=0$ a.e. in $\Omega$. Hence, $u(x)\in [0,\infty)^2$ for a.a. $x\in\Omega$ as claimed.
	
\medskip

\noindent\textbf{Step~3.} It remains to prove (ii). We thus assume that $\max\{U_1,U_2\} \le \rho$ a.e. in $\Omega$ and consider a weak solution $u\in H^1(\Omega,\mathbb{R}^2)$ to~\eqref{ap3}. As $v=((u_1-\rho)_+, (u_2-\rho)_+)$  belongs to $H^1(\Omega,\mathbb{R}^2)$, we deduce from~\eqref{ap3} that 
\begin{align*}
	\int_\Omega \Big[ \sum_{j=1}^2 (u_j-U_j) (u_j-\rho)_+ + \tau \sum_{i=1}^N\sum_{j,k=1}^2 b_{jk}(u) \partial_i u_k \partial_i (u_j-\rho)_+ \Big]\ \mathrm{d}x=0\,.
\end{align*}
On the one hand,
\begin{equation*}
	u_j-U_j \ge u_j - \rho \;\text{ a.e. in}\; \Omega\,, \qquad j=1,2\,,
\end{equation*}
so that
\begin{equation*}
	(u_j-U_j) (u_j-\rho)_+ \ge (u_j-\rho) (u_j-\rho)_+ = (u_j-\rho)_+^2 \;\text{ a.e. in }\; \Omega\,, \qquad j=1,2\,.
\end{equation*}
On the other hand, we infer from \eqref{ap10'} that, for $1\le i \le N$, 
\begin{align*}
	b_{11}(u) \partial_i u_1 \partial_i(u_1-\rho)_+ & = b_{11}(u) \mathbf{1}_{(\rho,\infty)}(u_1) |\partial_i u_1|^2\geq  0\,, \\
	b_{12}(u) \partial_i u_2 \partial_i(u_1-\rho)_+ & = b_{12}(u) \mathbf{1}_{(\rho,\infty)}(u_1) \partial_i u_1 \partial_i u_2 = 0\,, \\
	b_{21}(u) \partial_i u_1 \partial_i(u_2-\rho)_+ & = b_{21}(u) \mathbf{1}_{(\rho,\infty)}(u_2) \partial_i u_1 \partial_i u_2 = 0\,, \\
	b_{22}(u) \partial_i u_2 \partial_i(u_2-\rho)_+ & = b_{22}(u) \mathbf{1}_{(\rho,\infty)}(u_2) |\partial_i u_2|^2 \geq0\,.
\end{align*}
Therefore, 
\begin{equation*}
	\sum_{j=1}^2 \int_\Omega (u_j-\rho)_+^2\, \mathrm{d}x\leq 0\,,
\end{equation*}
from which we deduce that $\max\{u_1,u_2\}\le\rho$ a.e. in $\Omega$.
\end{proof}

%%%%%%%%%%%%%%%%
%%%%%%%%%%%%%%%%
\subsection{A regularised system}\label{sec2.2}
%%%%%%%%%%%%%%%%
%%%%%%%%%%%%%%%%

We now introduce the two-parameter approximation of~\eqref{ex1} on which the subsequent analysis relies. Specifically, given $\rho>1$, we define 
\begin{equation*}
\alpha_\rho(z):=
\left\{
\begin{array}{clll}
	0&,& \quad z\leq 0,\\
	z&,& \quad 0\leq z\leq \rho-1,\\
	(\rho-1)(\rho-z)&,& \quad \rho-1\leq z\leq \rho,\\
	0,&,&\quad z\geq\rho,
\end{array}
\right.
\end{equation*}
and observe that $\alpha_\rho\in BC(\mathbb{R})$ with 
	\begin{equation*}
		0\le \alpha_\rho(z)\le \min\{\rho , z_+\}\,, \qquad z\in\mathbb{R}\,.
\end{equation*} 
Next, for $\varepsilon\in (0,1)$ and $X\in\mathbb{R}^2$, we set
\begin{equation*}
M_{\varepsilon}^\rho(X)  = (m_{\varepsilon, jk}^\rho(X))_{1\le j,k\le 2}:= \varepsilon I_2 + \lambda_\varepsilon((X_{1,+},X_{2,+}))M^\rho(X),
\end{equation*}
where 
\begin{equation}\label{momar}
	M^\rho(X) =  (m_{jk}^\rho(X))_{1\le j,k\le 2} :=  
	\begin{pmatrix}
		a\alpha_\rho(X_1) & b\alpha_\rho(X_1) \\
		c\alpha_\rho(X_2) & d\alpha_\rho(X_2)
	\end{pmatrix}\,, \qquad X\in\mathbb{R}^2\,,
\end{equation}
and 
\begin{equation*}
	\lambda_\varepsilon(X) := \frac{2}{1+\exp{[\varepsilon (X_1+X_2)}]}\,, \qquad X\in\mathbb{R}^2\,.
\end{equation*}
Note that $(M^\rho)_{\rho>1}$ converges to $M$, defined in~\eqref{x1}, locally uniformly in $[0,\infty)^2$ as $\rho\to\infty$, while~$(\lambda_\varepsilon)_{\varepsilon\in (0,1)}$ converges to $1$ locally uniformly in $\mathbb{R}^2$ as $\varepsilon\to 0$. 
In fact, for $R>0$,
\begin{equation}
	|\lambda_\varepsilon(X) - 1| \le 2 R \varepsilon\,, \qquad X\in [-R,R]^2\,. \label{UL}
\end{equation}

The outcome of this section is that, given $\tau>0$, $\varepsilon\in (0,1)$, $\varrho>1$, and $U\in L_{\infty,+}(\Omega,\mathbb{R}^2)$, there is a weak solution $u_\varepsilon^\rho\in H^1(\Omega,\mathbb{R}^2)\cap L_{\infty,+}(\Omega,\mathbb{R}^2)$ to 
\begin{equation*}
	u_\varepsilon^\rho - \tau \sum_{i=1}^N\partial_i \big( M_\varepsilon^\rho(u_\varepsilon^\rho) \partial_i u_\varepsilon^\rho \big) = U \;\;\text{ in }\;\; \Omega\,,
\end{equation*}
which satisfies an appropriate weak version of~\eqref{ex2}, as stated below. The next lemma is actually the building block of the proof of \Cref{P:1}. 

%%%%%%%%%%%%%%%%
\begin{lemma}\label{lem.ex2}
	Given  $\tau>0$, $U=(F,G)\in L_{\infty,+}(\Omega,\mathbb{R}^2)$, $\varepsilon\in (0,1)$, and $\rho\geq \max\{1,\|F\|_\infty,\|G\|_\infty\}$, there is a weak solution $u_{\varepsilon}^\rho=(u_{\varepsilon,1}^\rho, u_{\varepsilon,2}^\rho)\in H^1(\Omega,\mathbb{R}^2)\cap L_{\infty,+}(\Omega,\mathbb{R}^2)$ to 
	\begin{equation}
		\int_\Omega \Big[ \langle u_{\varepsilon}^\rho , v \rangle + \tau \sum_{i=1}^N\langle M_{\varepsilon }^\rho(u_{\varepsilon}^\rho) \partial_i u_{\varepsilon}^\rho , \partial_i v \rangle \Big]\ \mathrm{d}x 
		= \int_\Omega \langle U , v \rangle\ \mathrm{d}x\,, \qquad v\in H^1(\Omega,\mathbb{R}^2)\,, \label{ex12a}
	\end{equation}
	which additionally satisfies
	\begin{align}
		\max\{\|u_{\varepsilon,1}^\rho\|_\infty,\,\|u_{\varepsilon,2}^\rho\|_\infty\}&\leq \rho\label{ex12a'},\\[1ex]
		\|u_{\varepsilon}^\rho\|_2&\leq C_0\|U\|_2,\label{ex12d}\\[1ex]
		\|\nabla u_{\varepsilon}^\rho\|_2&\leq C_1(\tau,\varepsilon)\|U\|_2\,.\label{ex12e}
	\end{align}
	
Moreover, given $n\ge 2$, there exists a constant $C(n)$
such that 
\begin{equation}
	\mathcal{E}_n(u_\varepsilon^\rho)\le \tau C(n) \frac{\rho^{n-1}}{e^{\varepsilon\rho}}\|\nabla u_\varepsilon^\rho\|_2^2\ + \mathcal{E}_n(U)\,. \label{ex12c}
\end{equation}
\end{lemma}
%%%%%%%%%%%%%%%%

\begin{proof}
Let $\varepsilon\in (0,1)$ and $\rho\geq \max\{1,\|F\|_\infty,\|G\|_\infty\}$. To deduce the existence result stated in \Cref{lem.ex2} from the already established \Cref{lem.ap1}, we first recast~\eqref{ex12a} in the form~\eqref{ap3}. First, owing to the definition of the function $\alpha_\rho$, the matrix $M_\varepsilon^\rho$ lies in $BC(\mathbb{R}^2,\mathbf{M}_2(\mathbb{R}))$ and satisfies 
\begin{subequations}\label{ex11}
	\begin{equation}
		0 \le m_{\varepsilon,jk}^\rho(X) \le \varepsilon + 2 \rho\max\{a,\, b,\, c,\, d\}\,, \qquad 1\le j,k\le 2\,, \ X\in\mathbb{R}^2\,, \label{ex11a}
\end{equation}
as well as 
\begin{equation}
	\begin{split}
		m_{\varepsilon,11}^\rho(X) \geq m_{\varepsilon,12}^\rho(X) & = 0\,, \qquad X\in (-\infty,0)\times \mathbb{R}\,, \\
		m_{\varepsilon,22}^\rho(X) \geq m_{\varepsilon,21}^\rho(X) & = 0\,, \qquad X\in \mathbb{R}\times (-\infty,0)\,.
	\end{split} \label{ex11b} 
\end{equation}
and
\begin{equation}
	\begin{split}
		m_{\varepsilon,11}^\rho(X) \geq m_{\varepsilon,12}^\rho(X) & = 0\,, \qquad X\in (\rho,\infty)\times \mathbb{R}\,, \\
		m_{\varepsilon,22}^\rho(X) \geq m_{\varepsilon,21}^\rho(X) & = 0\,, \qquad X\in \mathbb{R}\times (\rho,\infty)\,.
	\end{split} \label{ex11b'} 
\end{equation}

Next, according to \cite{DGJ1997}, it is natural to use the Hessian matrix of the convex function $\Phi_2$ to symmetrize~\eqref{ex12a}. We thus set
\begin{equation*}
	S:=\frac{bd}{2}D^2\Phi_2 =
	\begin{pmatrix}
		ac& bc \\[1ex]
		bc & bd
	\end{pmatrix}
\end{equation*}
and observe that $S$ is symmetric and positive definite by~\eqref{condabcd}. In addition, for all~${X\in\mathbb{R}^2}$,
\begin{equation*}
	SM_\varepsilon^\rho(X) = \varepsilon S + \lambda_\varepsilon((X_{1,+},X_{2,+})) SM^\rho(X)
\end{equation*}
with
\begin{equation*}
	SM^\rho(X) = \begin{pmatrix}
		a^2 c \alpha_\rho(X_1) + b c^2 \alpha_\rho(X_2) & abc \alpha_\rho(X_1) + bcd \alpha_\rho(X_2) \\
		& \\
		abc \alpha_\rho(X_1) + bcd \alpha_\rho(X_2)  & b^2 c \alpha_\rho(X_1) + bd^2 \alpha_\rho(X_2)
	\end{pmatrix} \in {\mathbf{Sym}_2(\mathbb{R})}\,.
\end{equation*}
Since $\mathrm{tr}(SM^\rho(X))\ge 0$ and 
\begin{equation*}
	\det(S M^\rho(X)) = \det(S) \det(M^\rho(X)) = bc (ad-bc)^2 \alpha_\rho(X_1) \alpha_\rho(X_2)\ge 0 
\end{equation*}
by \eqref{condabcd}, the matrix $SM^\rho(X)$ is positive semidefinite, so that the matrix  $SM_{\varepsilon}^\rho(X)$ belongs to $\mathbf{SPD}_2(\mathbb{R})$ for all~${X\in\mathbb{R}^2}$ with 
\begin{equation}
	\langle SM_{\varepsilon}^\rho(X)\xi,\xi\rangle  \ge \varepsilon \langle S\xi,\xi \rangle \ge \varepsilon \frac{\det(S)}{\mathrm{tr}(S)} |\xi|^2 = \varepsilon \frac{bc(ad-bc)}{ac+bd} |\xi|^2\,, \qquad \xi\in\mathbb{R}^2\,. \label{ex11c}
\end{equation}
\end{subequations}
According to the properties~\eqref{ex11}, we are now in a position to apply Lemma~\ref{lem.ap1} (with $A=S$ and~${B=M_{\varepsilon}^\rho}$) and deduce that there is a solution~${u_{\varepsilon}^\rho\in H^1(\Omega,\mathbb{R}^2)\cap L_{\infty,+}(\Omega,\mathbb{R}^2)}$ to~\eqref{ex12a} which satisfies~\eqref{ex12a'}. Moreover, it follows from~\eqref{ex12a} (with $v=Su_\varepsilon^\rho\in H^1(\Omega,\mathbb{R}^2)$), \eqref{ex11c}, and the positive definiteness of $S$,
\begin{equation*}
	\langle S\xi,\xi\rangle \ge \frac{bc(ad-bc)}{ac+bd} |\xi|^2\,, \qquad \xi\in\mathbb{R}^2\,,
\end{equation*}
that 
\begin{align*}
	\|SU\|_2 \|u_\varepsilon^\rho\|_2 \ge \int_\Omega \langle SU, u_\varepsilon^\rho \rangle\ \mathrm{d}x & = \int_\Omega \Big[ \langle u_\varepsilon^\rho , S u_\varepsilon^\rho \rangle + \tau \sum_{i=1}^N\langle M_\varepsilon^\rho(u_\varepsilon^\rho) \partial_i u_\varepsilon^\rho , \partial_i S u_\varepsilon^\rho \rangle \Big]\ \mathrm{d}x \\
	& = \int_\Omega \Big[ \langle Su_\varepsilon^\rho ,  u_\varepsilon^\rho \rangle + \tau \sum_{i=1}^N\langle S M_\varepsilon^\rho(u_\varepsilon^\rho) \partial_i u_\varepsilon^\rho , \partial_i u_\varepsilon^\rho \rangle \Big]\ \mathrm{d}x \\
	& \ge \frac{bc(ad-bc)}{ac+bd} \left( \|u_\varepsilon^\rho\|_2^2 + \tau \varepsilon \|\nabla u_\varepsilon^\rho\|_2^2 \right)\,.
\end{align*}
Owing to \eqref{condabcd}, we conclude that the estimates~\eqref{ex12d} and~\eqref{ex12e} are satisfied.   
	
It remains to establish the estimate~\eqref{ex12c}. Let $n\ge 2$. Since ${u_{\varepsilon}^\rho\in H^1(\Omega,\mathbb{R}^2)\cap L_\infty(\Omega,\mathbb{R}^2)}$, the vector field $D\Phi_n(u_\varepsilon^\rho)$ belongs to $H^1(\Omega,\mathbb{R}^2)$  and we infer from \eqref{ex12a} (with $v=D\Phi_n(u_\varepsilon^\rho)$) that 
\begin{equation}\label{pro1}
	\int_\Omega \Big[ \langle u_{\varepsilon}^\rho -U, D\Phi_n(u_\varepsilon^\rho) \rangle + \tau \sum_{i=1}^N\langle M_{\varepsilon }^\rho(u_{\varepsilon}^\rho) \partial_i u_{\varepsilon}^\rho , \partial_i D\Phi_n(u_\varepsilon^\rho)\rangle \Big]\ \mathrm{d}x 
		= 0\,.
\end{equation}
On the one hand, the convexity of $\Phi_n$  implies that 
\begin{equation}\label{pro2}
	\int_\Omega  \langle u_{\varepsilon}^\rho -U, D\Phi_n(u_\varepsilon^\rho) \rangle\ \mathrm{d}x \ge \int_\Omega [\Phi_n( u_{\varepsilon}^\rho) - \Phi_n(U) ]\ \mathrm{d}x =\mathcal{E}_n(u_\varepsilon^\rho) - \mathcal{E}_n(U)\,. 
\end{equation}
On the other hand, using the symmetry and the positive semidefiniteness of the matrix $ D^2\Phi_n(u_\varepsilon^\rho)$, see Lemma~\ref{lem.p2}, we have
\begin{align}
	\tau \sum_{i=1}^N\int_\Omega \langle M_{\varepsilon}^\rho(u_{\varepsilon}^\rho) \partial_i u_{\varepsilon}^\rho , \partial_i D\Phi_n(u_\varepsilon^\rho)\rangle\ \mathrm{d}x &=\tau \sum_{i=1}^N\int_\Omega \langle M_{\varepsilon }^\rho(u_{\varepsilon}^\rho) \partial_i u_{\varepsilon}^\rho ,   D^2\Phi_n(u_\varepsilon^\rho)\partial _i u_\varepsilon^\rho\rangle\ \mathrm{d}x \nonumber\\[1ex]
	&=\tau \sum_{i=1}^N\int_\Omega \langle D^2\Phi_n(u_\varepsilon^\rho) M_{\varepsilon }^\rho(u_{\varepsilon}^\rho) \partial_i u_{\varepsilon}^\rho ,   \partial _i u_\varepsilon^\rho\rangle\ \mathrm{d}x \nonumber\\[1ex]
	&=\tau\varepsilon \sum_{i=1}^N\int_\Omega \langle D^2\Phi_n(u_\varepsilon^\rho)  \partial_i u_{\varepsilon}^\rho ,   \partial _i u_\varepsilon^\rho\rangle\ \mathrm{d}x \nonumber\\[1ex]
	&\quad + \tau  \sum_{i=1}^N\int_\Omega \lambda_\varepsilon(u_\varepsilon^\rho)  \langle D^2\Phi_n(u_\varepsilon^\rho) M^\rho(u_\varepsilon^\rho) \partial_i u_{\varepsilon}^\rho ,   \partial _i u_\varepsilon^\rho\rangle\ \mathrm{d}x \nonumber\\[1ex]
	&\geq \tau \sum_{i=1}^N\int_\Omega	\lambda_\varepsilon(u_\varepsilon^\rho) \langle D^2\Phi_n(u_\varepsilon^\rho) M^\rho(u_\varepsilon^\rho) \partial_i u_{\varepsilon}^\rho ,   \partial _i u_\varepsilon^\rho\rangle\ \mathrm{d}x\,. \label{pro3}
\end{align}		 	
Since $S_n(u_\varepsilon^\rho) := D^2\Phi_n(u_\varepsilon^\rho) M(u_\varepsilon^\rho)$ is positive semidefinite by Lemma~\ref{lem.p3}, we further have
\begin{align}
		&\hspace{-1cm} \tau \sum_{i=1}^N\int_\Omega	 \lambda_\varepsilon(u_\varepsilon^\rho) \langle D^2\Phi_n(u_\varepsilon^\rho) M^\rho(u_\varepsilon^\rho) \partial_i u_{\varepsilon}^\rho ,   \partial _iu_\varepsilon^\rho\rangle\ \mathrm{d}x \nonumber	\\[1ex]
		& = \tau \sum_{i=1}^N\int_\Omega	 \lambda_\varepsilon(u_\varepsilon^\rho) \langle D^2\Phi_n(u_\varepsilon^\rho) M(u_\varepsilon^\rho) \partial_i u_{\varepsilon}^\rho ,   \partial _iu_\varepsilon^\rho\rangle\ \mathrm{d}x \nonumber \\[1ex]
		&\qquad + \tau \sum_{i=1}^N\int_\Omega	 \lambda_\varepsilon(u_\varepsilon^\rho) \langle D^2\Phi_n(u_\varepsilon^\rho)\big[ M^\rho(u_\varepsilon^\rho) - M(u_\varepsilon^\rho)\big] \partial_i u_{\varepsilon}^\rho ,   \partial _i u_\varepsilon^\rho\rangle\ \mathrm{d}x \nonumber \\[1ex]
		&\geq \tau \sum_{i=1}^N\int_\Omega	 \lambda_\varepsilon(u_\varepsilon^\rho) \langle D^2\Phi_n(u_\varepsilon^\rho) \big[ M^\rho(u_\varepsilon^\rho) - M(u_\varepsilon^\rho)\big] \partial_i u_{\varepsilon}^\rho , \partial _i u_\varepsilon^\rho\rangle\ \mathrm{d}x\,. \label{pro4}
\end{align}
Taking now advantage of the fact that $0 \le  u_{\varepsilon,j}^\rho\leq \rho$ a.e.  in $\Omega$ for $j=1,\, 2$ by~\eqref{ex12a'}, we further have
\begin{align*}
	&\hspace{-0.5cm}\Big| \tau \sum_{i=1}^N\int_\Omega	 \lambda_\varepsilon(u_\varepsilon^\rho) \langle D^2\Phi_n(u_\varepsilon^\rho) \big[ M^\rho(u_\varepsilon^\rho) - M(u_\varepsilon^\rho)\big] \partial_i u_{\varepsilon}^\rho , \partial _i u_\varepsilon^\rho\rangle\ \mathrm{d}x\Big| \nonumber\\[1ex]
	& \le 2\tau \max\{a,\,b,\,c,\,d\} \|D^2\Phi_n\|_{L_\infty((0,\rho)^2)} \sum_{j=1}^2 \int_\Omega \lambda_\varepsilon(u_\varepsilon^\rho)  |\alpha_\rho(u_{\varepsilon,j}^\rho) - u_{\varepsilon,j}^\rho|\, |\nabla u_\varepsilon^\rho|^2 \ \mathrm{d}x \nonumber\\[1ex]
	& \leq 4\tau \max\{a,\,b,\,c,\,d\}\kappa_n \rho^{n-2} \sum_{j=1}^2 \int_{\{\rho-1\leq u_{\varepsilon,j}^\rho\leq\rho\}} \frac{|\alpha_\rho(u_{\varepsilon,j}^\rho) - u_{\varepsilon,j}^\rho|}{1+\exp(\varepsilon u_{\varepsilon,j}^\rho)} |\nabla u_\varepsilon^\rho|^2 \ \mathrm{d}x\,,
\end{align*} 
where $\kappa_n\in\mathbb{R}$ is a positive constant such that
\begin{equation*}
|D^2\Phi_n(X)|\leq \kappa_n(X_1^{n-2}+X_2^{n-2})\qquad\text{for all $X\in[0,\infty)^2$.}
\end{equation*}
Owing to the definition of $\alpha_\rho$, we further obtain
\begin{align}
	&\hspace{-0.5cm}\Big| \tau \sum_{i=1}^N\int_\Omega	 \lambda_\varepsilon(u_\varepsilon^\rho) \langle D^2\Phi_n(u_\varepsilon^\rho) \big[ M^\rho(u_\varepsilon^\rho) - M(u_\varepsilon^\rho)\big] \partial_i u_{\varepsilon}^\rho , \partial _i u_\varepsilon^\rho\rangle\ \mathrm{d}x\Big| \nonumber\\[1ex]
	&\leq 4 \tau \max\{a,\,b,\,c,\,d\}\kappa_n \rho^{n-2} \sum_{j=1}^2 \int_{\{\rho-1\leq u_{\varepsilon,j}^\rho\leq\rho\}}	\frac{\rho}{1+e^{\varepsilon(\rho-1)}} |\nabla u_\varepsilon^\rho|^2\ \mathrm{d}x \nonumber\\[1ex]
	&\leq 8 e \tau \max\{a,\,b,\,c,\,d\}\kappa_n \rho^{n-1} e^{-\varepsilon\rho} \|\nabla u_\varepsilon^\rho\|_2^2\,. \label{pro5}
\end{align}
The desired estimate~\eqref{ex12c} is now a straightforward consequence of the relations \eqref{pro1}-\eqref{pro5}. 
\end{proof}

%%%%%%%%%%%%%%%%
%%%%%%%%%%%%%%%%
\subsection{A regularised system: $\rho\to\infty$}\label{sec2.3}
%%%%%%%%%%%%%%%%
%%%%%%%%%%%%%%%%

We next study the cluster points as $\rho\to\infty$ of the family~${\{u_{\varepsilon}^\rho\,:\,  \rho\geq \max\{1,\|F\|_\infty,\|G\|_\infty\}\}}$ provided in \Cref{lem.ex2}, the parameter $\varepsilon\in (0,1)$ being held fixed.

%%%%%%%%%%%%%%%%
\begin{lemma}\label{lem.ex3}
Given $\tau>0$,  $U=(F,G)\in L_{\infty,+}(\Omega,\mathbb{R}^2)$, and $\varepsilon\in (0,1)$, there exist a sequence~$(\rho_l)_{l\ge 1}$   and	
a function~${u_{\varepsilon} = (u_{\varepsilon,1}, u_{\varepsilon,2}) \in H^1(\Omega,\mathbb{R}^2)\cap L_{\infty,+}(\Omega,\mathbb{R}^2)}$ such that  $\rho_l\to\infty$ and
\begin{align}
	u_\varepsilon^{\rho_l}&\to u_\varepsilon\qquad\text{in $L_p(\Omega,\mathbb{R}^2)$ for all $p\in [1,\infty)$ and pointwise a.e. in $\Omega$}\,,\label{conv1}\\[1ex]
	\nabla u_\varepsilon^{\rho_l}&\rightharpoonup \nabla u_\varepsilon\qquad \text{in $L_2(\Omega,\mathbb{R}^{2N})$}\,.\label{conv3}
\end{align}
Moreover, $u_\varepsilon$ solves the equation
\begin{equation}
	\int_\Omega \Big[ \langle u_{\varepsilon} , v \rangle + \tau \sum_{i=1}^N\langle M_{\varepsilon}(u_{\varepsilon}) \partial_i u_{\varepsilon} , \partial_i v \rangle \Big]\ \mathrm{d}x = \int_\Omega \langle U , v \rangle\ \mathrm{d}x\,, \qquad v\in H^1(\Omega,\mathbb{R}^2)\,, \label{ex12aa}
\end{equation}
where
\begin{equation*}
	M_{\varepsilon}(X)  = (m_{\varepsilon, jk}(X))_{1\le j,k\le 2}:= \varepsilon I_2 + \lambda_\varepsilon((X_{1,+},X_{2,+}))M(X),
\end{equation*}
with $M(X)$ defined in \eqref{x1}, and, for  each $n\ge 2$, we have
\begin{equation}
	\mathcal{E}_n(u_\varepsilon)\le \mathcal{E}_n(U)\,. \label{ex12ba}
\end{equation}
	Furthermore,
\begin{equation}
	\min\left\{ 1 , \frac{c}{d}\right\} \|u_{\varepsilon,1}+u_{\varepsilon,2}\|_\infty \le \max\left\{1,\frac{a}{b} \right\}\|F+G\|_\infty\,. \label{ex12bac}
	\end{equation}
\end{lemma}
%%%%%%%%%%%%%%%%

\begin{proof}
Recalling \eqref{ex12d}-\eqref{ex12e}, we deduce that  $(u_\varepsilon^\rho)_\rho$ is bounded in $H^1(\Omega,\mathbb{R}^2)$. Moreover, since 
\begin{equation}
	\frac{\varepsilon^n z^n}{n!} \le e^{\varepsilon z}\,, \qquad z\in [0,\infty)\,, \quad n\ge 1\,, \label{x6}
\end{equation}
the estimates~\eqref{ex12e} and \eqref{ex12c}, along with \Cref{lelfb}, ensure that $(u_\varepsilon^\rho)_\rho$ is bounded in $L_n(\Omega,\mathbb{R}^2)$ for any integer~$n\geq 2$ (with an $\varepsilon$-dependent bound). We may then use a Cantor diagonal process, together with Rellich-Kondrachov' theorem and an interpolation argument, to deduce the convergence~\eqref{conv1} and~\eqref{conv3} along a sequence $\rho_l\to\infty$, as well as the componentwise non-negativity of~$u_\varepsilon$.
	
Since $\Phi_n$ is convex on $[0,\infty)^2$ for all $n\geq 2$, see Lemma~\ref{lem.p2}, it follows from \eqref{ex12e}, \eqref{ex12c}, \eqref{conv1}, and \eqref{x6} that~\eqref{ex12ba} holds true. Using once more \Cref{lelfb}, we infer from~\eqref{ex12ba} that
\begin{equation*}
	\|c u_{\varepsilon,1} + d u_{\varepsilon, 2}\|_n \le \frac{d}{b} \|a F + b G\|_n
\end{equation*}
for all $n\ge 2$. 
Passing to the limit $n\to\infty$ in the above inequality, we deduce that $u_\varepsilon\in L_\infty(\Omega,\mathbb{R}^2)$ satisfies~\eqref{ex12bac}. 

Let us now consider $v\in H^1(\Omega,\mathbb{R}^2)$. Since  \eqref{conv1} and \eqref{conv3} imply that  
\begin{equation*}
	\lim_{l\to\infty} \int_\Omega  \langle u_{\varepsilon}^{\rho_l} , v \rangle\ \mathrm{d}x = \int_\Omega  \langle u_{\varepsilon}, v \rangle\ \mathrm{d}x
	\qquad\text{and}\qquad \lim_{l\to\infty} \int_\Omega   \langle  \partial_i u_{\varepsilon}^{\rho_l} , \partial_i v \rangle \ \mathrm{d}x 
	\to\int_\Omega   \langle  \partial_i u_{\varepsilon} , \partial_i v \rangle \ \mathrm{d}x 
\end{equation*}
for $1\le i \le N$, the identity~\eqref{ex12aa} is satisfied provided that 
\begin{equation}
	\lim_{l\to\infty} \int_\Omega    \lambda_\varepsilon(u_{\varepsilon}^{\rho_l}) \langle M^{\rho_l}(u_{\varepsilon}^{\rho_l}) \partial_i u_{\varepsilon}^{\rho_l} , \partial_i v \rangle \ \mathrm{d}x = \int_\Omega   \lambda_\varepsilon(u_{\varepsilon}) \langle M(u_{\varepsilon}) \partial_i u_{\varepsilon} , \partial_i v \rangle\ \mathrm{d}x \label{ex12ac}
\end{equation}
for each $1\leq i\leq N$.
 To prove~\eqref{ex12ac}, we observe that, for $1\leq i\leq N$ and $j\in\{1,2\}$,
\begin{equation}
	\int_\Omega \lambda_\varepsilon(u_\varepsilon^{\rho_l}) \langle M^{\rho_l}(u_{\varepsilon}^{\rho_l}) \partial_i u_{\varepsilon}^{\rho_l} , \partial_i v \rangle \ \mathrm{d}x = \int_\Omega \lambda_\varepsilon(u_\varepsilon^{\rho_l}) \langle M^{\rho_l}(u_{\varepsilon}^{\rho_l})^t \partial_i v , \partial_i u_{\varepsilon}^{\rho_l} \rangle \ \mathrm{d}x \label{x7}
\end{equation}	
with
\begin{align*}
	\left| \lambda_\varepsilon(u_\varepsilon^{\rho_l}) \sum_{k=1}^2 m_{kj}^{\rho_l}(u_{\varepsilon}^{\rho_l}) \partial_i v_k \right| 
	& \le 2\max\{a,b,c,d\} \frac{u_{\varepsilon,1}^{\rho_l} + u_{\varepsilon,2}^{\rho_l}}{1+ \exp[\varepsilon (u_{\varepsilon,1}^{\rho_l} + u_{\varepsilon,2}^{\rho_l})]} |\partial_i v| \\
	& \hspace{3cm} \le \frac{2\max\{a,b,c,d\} }{\varepsilon} |\partial_i v| \qquad\text{a.e. in $\Omega$}\,,
\end{align*}
by \eqref{ex12bac} and \eqref{x6}, and
\begin{equation*}
	\lim_{l\to\infty} \lambda_\varepsilon(u_\varepsilon^{\rho_l}) \sum_{k=1}^2 m_{kj}^{\rho_l}(u_{\varepsilon}^{\rho_l}) \partial_i v_k 
	= \lambda_\varepsilon(u_\varepsilon) \sum_{k=1}^2 m_{ kj}(u_{\varepsilon}) \partial_i v_k \qquad\text{a.e. in $\Omega$}\,,
\end{equation*}
by~\eqref{momar}, the pointwise  almost everywhere convergence in $\Omega$ established in \eqref{conv1}, and the properties of $\alpha_{\rho_l}$. 
Lebesgue's dominated convergence theorem then guarantees that
\begin{equation*}
	\lim_{l\to\infty} \left\| \lambda_\varepsilon(u_\varepsilon^{\rho_l}) \sum_{k=1}^2 m_{kj}^{\rho_l}(u_{\varepsilon}^{\rho_l}) \partial_i v_k 
	- \lambda_\varepsilon(u_\varepsilon) \sum_{k=1}^2 m_{kj}(u_{\varepsilon}) \partial_i v_k \right\|_2 = 0\,.
\end{equation*}  
Combining the above convergence with~\eqref{conv3}, allows us to pass to the limit as $l\to\infty$ in~\eqref{x7} and find
\begin{align*}
	\lim_{l\to\infty} \int_\Omega \lambda_\varepsilon(u_\varepsilon^{\rho_l}) \langle M^{\rho_l}(u_{\varepsilon}^{\rho_l}) \partial_i u_{\varepsilon}^{\rho_l} , \partial_i v \rangle \ \mathrm{d}x & = \int_\Omega \lambda_\varepsilon(u_\varepsilon) \langle M(u_{\varepsilon})^t \partial_i v , \partial_i u_{\varepsilon} \rangle \ \mathrm{d}x \\
	& = \int_\Omega \lambda_\varepsilon(u_\varepsilon) \langle M(u_{\varepsilon}) \partial_i u_{\varepsilon} , \partial_i v \rangle \ \mathrm{d}x
\end{align*}
for $1\le i \le N$, which proves~\eqref{ex12ac}. We have thus shown that $u_\varepsilon$ solves~\eqref{ex12aa} and thereby completed the proof of \Cref{lem.ex3}.
\end{proof}

We next show that the entropy functional $\mathcal{E}_1$ evaluated at the function~$u_\varepsilon$ identified in \Cref{lem.ex3} is dominated by $\mathcal{E}_1(U)$ and that the associated dissipation term $\mathcal{E}_1(U)-\mathcal{E}_1(u_\varepsilon)$ provides a control on the gradient of $u_\varepsilon$ which is essential when considering the limit $\varepsilon\to 0$. 

%%%%%%%%%%%%%%%%
\begin{lemma}\label{lem.ex4}
Let $\tau>0$, $U=(F,G)\in L_{\infty,+}(\Omega,\mathbb{R}^2)$,  and  $\varepsilon\in (0,1)$. 	
The function
\begin{equation*}
	u_{\varepsilon}=(u_{\varepsilon,1},u_{\varepsilon,2})\in H^1(\Omega,\mathbb{R}^2)\cap L_{\infty,+}(\Omega,\mathbb{R}^2)
\end{equation*} 
identified in \Cref{lem.ex3} satisfies 
\begin{equation*}
	\mathcal{E}_1(u_{\varepsilon}) + \frac{ \tau}{a} \int_\Omega \lambda_\varepsilon(u_{\varepsilon}) \big[  |\nabla  (a u_{\varepsilon,1} + \Theta_1  u_{\varepsilon,2})|^2 + \Theta_2 |\nabla u_{\varepsilon,2} |^2 \big]\ \mathrm{d}x \le \mathcal{E}_1(U)\,.
	\end{equation*}
\end{lemma}
%%%%%%%%%%%%%%%%

\begin{proof}
Let $\eta\in (0,1)$. Then $\left( \ln{(u_{\varepsilon,1}+\eta)}, (b^2/ad)\ln{(u_{\varepsilon,2}+\eta) } \right)\in H^1(\Omega,\mathbb{R}^2)$ and we infer from~\eqref{ex12aa} that
\begin{equation}\label{ex19}
	0  = \int_\Omega \Big[ (u_{\varepsilon,1} - U_1) \ln{(u_{\varepsilon,1}+\eta)} + \frac{b^2}{ad} (u_{\varepsilon,2} - U_2)\ln{(u_{\varepsilon,2}+\eta)} \Big]\ \mathrm{d}x +D(\eta)\,,
\end{equation}
where
\begin{align*}
	D(\eta)&:= \tau \int_\Omega \sum_{i=1}^N\big( m_{\varepsilon,11}(u_{\varepsilon}) \partial_i u_{\varepsilon,1} + m_{\varepsilon,12}(u_{\varepsilon}) \partial_i u_{\varepsilon,2} \big)
		\frac{\partial_i u_{\varepsilon,1}}{u_{\varepsilon,1}+\eta}\ \mathrm{d}x \\[1ex]
	& \qquad + \frac{\tau b^2}{ad} \int_\Omega \sum_{i=1}^N\big( m_{\varepsilon,21}(u_{\varepsilon}) \partial_i u_{\varepsilon,1} + m_{\varepsilon,22}(u_{\varepsilon}) \partial_i u_{\varepsilon,2} \big) 
	\frac{\partial_i u_{\varepsilon,2}}{u_{\varepsilon,2}+\eta}\ \mathrm{d}x \,.
\end{align*}
Since $L(r)=r\ln r-r+1$ is convex on $[0,\infty)$ with $L'(r)=\ln{r}$, the first term on the right-hand side of~\eqref{ex19} can be estimated as follows   
\begin{align*}
	& \int_\Omega \Big[ (u_{\varepsilon,1} - U_1) \ln{(u_{\varepsilon,1}+\eta)} + \frac{b^2}{ad} (u_{\varepsilon,2} - U_2)\ln{(u_{\varepsilon,2}+\eta)} \Big]\ \mathrm{d}x \\[1ex]
	& \qquad \ge \int_\Omega \Big[ (L(u_{\varepsilon,1}+\eta) - L(U_1+\eta)) + \frac{b^2}{ad}(L(u_{\varepsilon,2}+\eta) - L(U_2+\eta)) \Big]\ \mathrm{d}x \\[1ex]
	& \qquad = \mathcal{E}_1((u_{\varepsilon,1}+\eta,u_{\varepsilon,2}+\eta))- \mathcal{E}_1((U_{1}+\eta,U_{2}+\eta))\,.
\end{align*}
Using the continuity of  $\Phi_1$ and the boundedness of $u_\varepsilon$, see~\eqref{ex12bac}, we deduce that
\begin{equation}
	\liminf_{\eta\to 0} \int_\Omega \left[ (u_{\varepsilon,1} - U_1) \ln{(u_{\varepsilon,1}+\eta)} + \frac{b^2}{ad} (u_{\varepsilon,2} - U_2)\ln{(u_{\varepsilon,2}+\eta)} \right]\ \mathrm{d}x  \ge \mathcal{E}_1(u_{ \varepsilon} )- \mathcal{E}_1(U)\,. \label{ex20}
\end{equation}
Next, recalling the definition of the matrix  $M_\varepsilon$, see \Cref{lem.ex3}, we have
\begin{align*}
	D(\eta) & = \tau\varepsilon\int_\Omega \left( \frac{|\nabla u_{\varepsilon,1}|^2}{u_{\varepsilon,1}+\eta} +\frac{b^2}{ad} \frac{|\nabla u_{\varepsilon,2}|^2}{u_{\varepsilon,2}+\eta} \right)\ \mathrm{d}x  \\
	& \qquad + \frac{\tau}{a} \int_\Omega\lambda_\varepsilon(u_{\varepsilon}) \big[  |\nabla(a u_{\varepsilon,1}+\Theta_1  u_{\varepsilon,2})|^2 +  \Theta_2|\nabla u_{\varepsilon,2} |^2 \big]\ \mathrm{d}x\\
	& \qquad - J_1(\eta) - J_2(\eta)\,,
\end{align*}
where
\begin{align*}
	J_1(\eta) & := \tau   \int_\Omega  \frac{\eta\lambda_\varepsilon(u_{\varepsilon}) }{u_{\varepsilon,1}+\eta}  \nabla u_{\varepsilon,1} \cdot \nabla(a u_{\varepsilon,1} + b u_{\varepsilon,2})  \ \mathrm{d}x\,, \\
	J_2(\eta) & := \frac{\tau b^2}{ad} \int_\Omega \frac{\eta\lambda_\varepsilon(u_{\varepsilon}) }{u_{\varepsilon,2}+\eta}  \nabla u_{\varepsilon,2} \cdot \nabla(c u_{\varepsilon,1} + d u_{\varepsilon,2})  \ \mathrm{d}x\,.
\end{align*}
Since $u_{\varepsilon}\in H^1(\Omega,\mathbb{R}^2)$  and $\nabla u_{\varepsilon,j}=0$ a.e. on the level set $\{x\in\Omega\ : \ u_{\varepsilon,j}=0\}$ for $j\in\{1,2\}$,   we have
\begin{align*}
	& \lim_{\eta\to 0} \frac{\eta\lambda_\varepsilon(u_{\varepsilon}) }{u_{\varepsilon,j}+\eta} \nabla u_{\varepsilon,j}=0\qquad \text{ a.e. in $\Omega$}\,, \\
	& \left| \frac{\eta\lambda_\varepsilon(u_{\varepsilon}) }{u_{\varepsilon,j}+\eta} \nabla u_{\varepsilon,j} \right|
	 \leq |\nabla u_{\varepsilon,j}|\qquad \text{ a.e. in $\Omega$}\,.
\end{align*}
Lebesgue's dominated convergence theorem ensures now that
\begin{equation*}
	\lim_{\eta\to 0} \left( J_1(\eta) + J_2(\eta) \right) = 0\,. 
\end{equation*}
This shows that
\begin{equation}
	\liminf_{\eta\to 0} D(\eta) \ge \frac{\tau}{a} \int_\Omega\lambda_\varepsilon(u_{\varepsilon}) \big[  |\nabla  (a u_{\varepsilon,1}+\Theta_1 u_{\varepsilon,2})|^2 +  \Theta_2|\nabla u_{\varepsilon,2} |^2 \big]\ \mathrm{d}x\,. \label{ex23}
\end{equation}
Passing to the limit $\eta\to 0$ in \eqref{ex19}, we get the  desired estimate in view of \eqref{ex20} and \eqref{ex23}.
\end{proof}

%%%%%%%%%%%%%%%%
%%%%%%%%%%%%%%%%
\subsection{A regularised system: $\varepsilon\to 0$}\label{sec2.4}
%%%%%%%%%%%%%%%%
%%%%%%%%%%%%%%%%

We complete this section with the proof of \Cref{P:1}.

\begin{proof}[Proof of \Cref{P:1}]
Consider $\tau>0$ and $U=(F,G)\in L_{\infty,+}(\Omega,\mathbb{R}^2)$. Given $\varepsilon\in (0,1)$, let 
\begin{equation*}
	u_\varepsilon =(u_{\varepsilon,1},u_{\varepsilon,2}) \in H^1(\Omega,\mathbb{R}^2)\cap  L_{\infty,+}(\Omega,\mathbb{R}^2)
\end{equation*} 
denote the weak solution to~\eqref{ex12aa} provided by \Cref{lem.ex3}. According to~\eqref{ex12bac}, 
\begin{equation}
	\max\{\|u_{\varepsilon,1}\|_\infty , \| u_{\varepsilon,2}\|_\infty \} \le\|u_{\varepsilon,1}+ u_{\varepsilon,2}\|_\infty  \le R_0 := \frac{d}{b} \frac{\max\{a,b\}}{\min\{c,d\}} \|F+G\|_\infty\,. \label{x8}
\end{equation}
Hence,
\begin{equation*}
	\lambda_\varepsilon(u_\varepsilon) \ge \frac{2}{1+e^{R_0}}\,,
\end{equation*}
a lower bound which, together with \Cref{lem.ex4} and the non-negativity of $\mathcal{E}_1$, ensures that
\begin{equation}
	\text{$(\nabla u_{\varepsilon})_{\varepsilon \in (0,1)}$ is bounded in~$L_2(\Omega,\mathbb{R}^{2N})$.} \label{x9}  
\end{equation}
We now infer from~\eqref{x8}, \eqref{x9}, Rellich-Kondrachov' theorem, 
an interpolation argument, and a Cantor diagonal process that there exist a function 
\begin{equation*}
	u=(f,g)\in H^1(\Omega,\mathbb{R}^2)\cap L_{\infty,+}(\Omega,\mathbb{R}^2)
\end{equation*} 
and a sequence $(\varepsilon_l)_{l\ge 1}$, with $\varepsilon_l\to 0$, such that
\begin{align}
	u_{\varepsilon_l}&\to u\qquad\text{in $L_p(\Omega,\mathbb{R}^2)$ for all $p\in [1,\infty)$}\,,\label{conv1'}\\[1ex]
	u_{\varepsilon_l}&\overset{*}\rightharpoonup u\qquad  \text{in $L_\infty(\Omega,\mathbb{R}^2)$}\,,\label{conv2'}\\[1ex]
	\nabla u_{\varepsilon_l}&\rightharpoonup \nabla u\qquad \text{in $L_2(\Omega,\mathbb{R}^{2N})$}\,.\label{conv3'}
\end{align}
An immediate consequence of~\eqref{ex12ba} and \eqref{conv1'} is the estimate~\eqref{ex2}.
 Since $\sqrt{\lambda_{\varepsilon_l}(u_{\varepsilon_l})}\to 1$ in~${L_\infty(\Omega)}$ by~\eqref{UL} and~\eqref{x8}, we conclude together with~\eqref{conv3'} that
\begin{align*}
	\sqrt{\lambda_{\varepsilon_l}(u_{\varepsilon_l})} \nabla \big(a u_{\varepsilon_l,1} + \Theta_1 u_{\varepsilon_l,2} \big) & 
	\rightharpoonup \nabla\big (a u_1 + \Theta_1 u_2\big) \qquad \text{in $L_2(\Omega,\mathbb{R}^{N})$}\,, \\[1ex]
	\sqrt{\Theta_2 \lambda_{\varepsilon_l}(u_{\varepsilon_l})} \nabla u_{\varepsilon_l,2} & \rightharpoonup \sqrt{\Theta_2}\nabla u_2 \qquad \text{in $L_2(\Omega,\mathbb{R}^{N}) $}\,.
\end{align*}
Moreover, the $L_\infty$-bound~\eqref{x8} and the convergence~\eqref{conv1'} imply that
\begin{equation*}
	\liminf_{l\to\infty} \mathcal{E}_1(u_{\varepsilon_l}) \geq \mathcal{E}_1(u)\,,
\end{equation*}
and the estimate~\eqref{ex2b} is now obtained by passing to $\liminf$ in the inequality reported in \Cref{lem.ex4} (with $\varepsilon$ replaced by $\varepsilon_l$).

Finally,~\eqref{conv1'}, along with~\eqref{x8} and the convergence property
\begin{equation*}
	\lim_{\varepsilon\to 0} \big| m_{\varepsilon,jk}(X) - m_{jk}(X) \big| = 0\,,
\end{equation*}
which is uniform with respect to $X\in [0,R_0]^2 $ and $1 \le j,k\le 2$, enables us  to use Lebesgue's dominated convergence theorem  to show that, for $v=(\varphi,\psi)\in H^1(\Omega,\mathbb{R}^2)$,
\begin{equation*}
	\lim_{l\to\infty} \big\| M_{\varepsilon_l}(u_{\varepsilon_l})^t \partial_i v - M(u)^t \partial_i v \big\|_2 = 0\,, \qquad 1\le i \le N\,.
\end{equation*}
Together with~\eqref{conv1'} and~\eqref{conv3'}, the above convergence allows us to let $\varepsilon_l\to 0$ in~\eqref{ex12aa} and conclude that  $u=(f,g)$ satisfies~\eqref{ex1}. This completes the proof of \Cref{P:1}.
\end{proof}

%%%%%%%%%%%%%%%%
%%%%%%%%%%%%%%%%
\section{Existence of bounded weak solutions}\label{sec3}
%%%%%%%%%%%%%%%%
%%%%%%%%%%%%%%%%

This section is devoted to the proof of \Cref{ThBWS}, which relies on rather classical arguments, besides the estimates derived in \Cref{P:1}, and proceeds along the lines of the proof of \cite[Theorem~1.2]{LM2021b}. As a first step, we use \Cref{P:1} to construct a family of piecewise constant functions $(u^\tau)_{\tau\in(0,1)}$ starting from the initial condition  $(f^{in},g^{in})\in L_{\infty,+}(\Omega,\mathbb{R}^2)$. More precisely, for~${\tau\in (0,1)}$, we set $u^\tau(0):=u_0^\tau $ and
\begin{equation}\label{x01}
	u^\tau(t)= u^\tau_{l}\,, \qquad t\in ((l-1)\tau, l\tau]\,, \qquad   l\in\mathbb{N}\setminus\{0\}\,,
\end{equation}
where the sequence $(u_{l}^\tau)_{l\geq 0}$ is defined as follows:
\begin{equation}
	\begin{split}
		&u_0^\tau = u^{in} := (f^{in},g^{in})\in L_{\infty,+}(\Omega,\mathbb{R}^2)\,, \\
		&u_{l+1}^\tau =(f_{l+1}^\tau,g_{l+1}^\tau)\in H^1(\Omega,\mathbb{R}^2)\cap L_{\infty,+}(\Omega,\mathbb{R}^2) \;\text{is the solution to \eqref{ex1}} \\
		&\text{with $ U=u_l^\tau=(f_l^\tau,g_l^\tau)$ constructed in \Cref{P:1} for $  l\ge 0$}\,.
	\end{split}\label{x02}
\end{equation}
In order to establish \Cref{ThBWS}, we show that the family $(u^\tau)_{\tau\in(0,1)}$  defined in~\eqref{x02} converges along a subsequence~${\tau_j\to0}$ towards a pair $u=(f,g)$ which fulfills all the requirements of \Cref{ThBWS}. 

\medskip

Below, $C$ and $(C_l)_{l\ge 0}$ denote various positive constants depending only on~${(a,\, b,\, c,\, d)}$ and $u^{in}$. Dependence upon additional parameters will be indicated explicitly.

\begin{proof}[Proof of \Cref{ThBWS}]
	Let $\tau\in (0,1)$ and let $u^\tau$ be defined in \eqref{x01}-\eqref{x02}.  Given   $l\ge 0$, we infer from \Cref{P:1} that
	\begin{subequations}\label{x03}
		\begin{align}
			\int_\Omega \Big( f_{l+1}^\tau \varphi + \tau f_{l+1}^\tau \nabla[ af_{l+1}^\tau + b g_{l+1}^\tau ] \cdot\nabla \varphi \Big)\ \mathrm{d}x & = \int_\Omega f_{l}^\tau \varphi\ \mathrm{d}x\,, \qquad \varphi\in H^1(\Omega)\,, \label{x03a} \\
			\int_\Omega \Big( g_{l+1}^\tau \psi + \tau   g_{l+1}^\tau \nabla[cf_{l+1}^\tau + dg_{l+1}^\tau] \cdot\nabla\psi \Big)\ \mathrm{d}x & = \int_\Omega g_{l}^\tau \psi\ \mathrm{d}x\,, \qquad \psi\in H^1(\Omega)\,. \label{x03b}
		\end{align}
	\end{subequations}
	Moreover,
	\begin{equation}
		\mathcal{E}_n(u_{l+1}^\tau) \le \mathcal{E}_n(u_{l}^\tau)\qquad\text{for $n\ge 2,$} \label{x04}
	\end{equation}
	 and we also have
	\begin{equation}
		\mathcal{E}_1(u_{l+1}^\tau) 
		+\frac{ \tau}{a} \int_\Omega \big[  |\nabla (a f_{l+1}^\tau + \Theta_1g_{l+1}^\tau)|^2+\Theta_2|\nabla g_{l+1}^\tau|^2\big]\ \mathrm{d}x \le\mathcal{E}_1(u_{l}^\tau)\,. \label{x05}
	\end{equation}
	It readily follows from \eqref{x01}, \eqref{x02}, \eqref{x04}, and \eqref{x05} that, for $t>0$,
	\begin{equation}
		\mathcal{E}_n(u^\tau(t)) \le \mathcal{E}_n (u^{in}) \,, \qquad \ n\ge 2\,, \label{x06}
	\end{equation}
and
	\begin{equation}
		\mathcal{E}_1(u^\tau(t))  +\frac{ 1}{a} \ \int_0^t \int_\Omega \big[  |\nabla (a f^\tau + \Theta_1g^\tau)|^2+\Theta_2|\nabla g^\tau|^2\big]\ \mathrm{d}x\mathrm{d}s \le 	\mathcal{E}_1(u^{in})\,. \label{x07}
	\end{equation}
	
	An immediate consequence of \eqref{x06} and Lemma~\ref{lelfb} is the estimate 
	\begin{equation*}
		\|f^\tau(t)+g^\tau(t)\|_n \le \frac{d}{b}\frac{\max\{a,\,b\}}{\min\{c,\, d\}}\|f^{in}+g^{in}\|_n\,, \qquad n\ge 2\,, \ t>0\,.
	\end{equation*}
	Letting $n\to\infty$ in the above inequality gives
	\begin{equation}
		\|f^\tau(t)+g^\tau(t)\|_\infty \le C_1 := \frac{d}{b}\frac{\max\{a,\,b\}}{\min\{c,\, d\}} \|f^{in}+g^{in}\|_\infty\,, \qquad t>0\,. \label{x08}
	\end{equation}
	Also, taking advantage of  the non-negativity of $\mathcal{E}_1$, we deduce from   \eqref{x07} that
	\begin{equation}
		\int_0^t \big[ \|\nabla f^\tau(s)\|_2^2 + \|\nabla g^\tau(s)\|_2^2 \big]\ \mathrm{d}s\le C_2 := \frac{a^2+2(\Theta_2+\Theta_1^2)}{a\Theta_2} \mathcal{E}_1(u^{in})\,, \qquad t>0\,. \label{x09}
	\end{equation}
	
	Next, for $l\ge 1$ and $t\in ((l-1)\tau,l\tau]$, we deduce from \eqref{x03a}, \eqref{x08}, and H\"older's inequality that, for $\varphi\in H^1(\Omega)$, 
	\begin{align*}
		\left| \int_\Omega \left( f^{\tau}(t+\tau) - f^\tau(t) \right) \varphi\ \mathrm{d}x \right| & = 
		\left| \int_{l\tau}^{(l+1) \tau} \int_\Omega f_{l+1}^\tau \nabla[ af_{l+1}^\tau + b g_{l+1}^\tau ] \cdot\nabla \varphi\ \mathrm{d}x\mathrm{d}s \right| \\
		& \le \int_{l\tau}^{(l+1) \tau} \| f^\tau(s)\|_\infty \|\nabla[ af^\tau(s) + b g^\tau(s) ]\|_2 \|\nabla\varphi\|_2\ \mathrm{d}s \\
		& \le C_1 \|\nabla\varphi\|_2 \int_{l\tau}^{(l+1) \tau} \|\nabla[ af^\tau(s) + b g^\tau(s) ]\|_2\ \mathrm{d}s \,.
	\end{align*}
	A duality argument then gives
	\begin{equation*}
		\| f^{\tau}(t+\tau) - f^\tau(t) \|_{(H^1)'} \le  C_1 \int_{l\tau}^{(l+1) \tau} \|\nabla[ af^\tau(s) + b g^\tau(s) ]\|_2\ \mathrm{d}s
	\end{equation*}
	for $t\in ((l-1)\tau,l\tau]$ and $l\ge 1$. Now, for $l_0\ge 2$ and $T\in ((l_0-1)\tau,l_0\tau]$, the above inequality, along with H\"older's inequality, entails that
	\begin{align*}
		\int_0^{T-\tau} \| f^{\tau}(t+\tau) - f^\tau(t) \|_{(H^1)'}^2\ \mathrm{d}t & \le \int_0^{(l_0-1)\tau} \| f^{\tau}(t+\tau) - f^\tau(t) \|_{(H^1)'}^2\ \mathrm{d}t \\
		&  =\sum_{l=1}^{l_0-1} \int_{(l-1)\tau}^{l\tau} \| f^{\tau}(t+\tau) - f^\tau(t) \|_{(H^1)'}^2\ \mathrm{d}t \\
		& \le C_1^2 \tau \sum_{l=1}^{l_0-1} \left( \int_{l\tau}^{(l+1) \tau} \|\nabla[ af^\tau(s) + b g^\tau(s) ]\|_2\ \mathrm{d}s \right)^2\\
		& \le C_1^2 \tau^2 \sum_{l=1}^{l_0-1} \int_{l\tau}^{(l+1) \tau} \|\nabla[ af^\tau(s) + b g^\tau(s) ]\|_2^2\ \mathrm{d}s \\
		& \le C_1^2 \tau^2 \int_{0}^{l_0 \tau} \|\nabla[ af^\tau(s) + b g^\tau(s) ]\|_2^2\ \mathrm{d}s\,.
	\end{align*}
	We then use \eqref{x09} (with $t=l_0\tau$) and Young's inequality to obtain
	\begin{align}
		\int_0^{T-\tau} \| f^{\tau}(t+\tau) - f^\tau(t) \|_{(H^1)'}^2\ \mathrm{d}t & \le C_1^2 \tau^2 
		\int_{0}^{l_0\tau} \left( 2 a^2 \|\nabla f^\tau(s)\|_2^2 + 2b^2 \|\nabla g^\tau(s)\|_2^2 \right)\ \mathrm{d}s \nonumber \\
		& \le C_3 \tau^2 \,, \label{x10}
	\end{align}
	with $C_3 := 2 (a^2+b^2)^2 C_1^2 C_2$. Similarly, 
	\begin{equation}
		\int_0^{T-\tau} \| g^{\tau}(t+\tau) - g^\tau(t) \|_{(H^1)'}^2\ \mathrm{d}t \le C_4 \tau^2 \,, \label{x11}
	\end{equation}
	with $C_4 := 2 (c^2+d^2) C_1^2 C_2$. 
	
	According to Rellich-Kondrachov' theorem, 
	$H^1(\Omega,\mathbb{R}^2)$ is compactly embedded in $L_2(\Omega,\mathbb{R}^2)$, while~${L_2(\Omega,\mathbb{R}^2)}$ is continuously (and compactly) 
	embedded in~$H^1(\Omega,\mathbb{R}^2)'$.
	Gathering \eqref{x08}-\eqref{x11}, we infer from \cite[Theorem~1]{DJ2012} that, for any~${T>0}$,
	\begin{equation}
		(u^\tau)_{\tau\in (0,1)} \;\text{ is relatively compact in }\; L_2((0,T)\times\Omega,\mathbb{R}^2)\,. \label{x12}
	\end{equation}
	Owing to \eqref{x08}, \eqref{x09}, and \eqref{x12}, we may use a Cantor diagonal argument to find a function
	\begin{equation*}
		u=(f,g)\in L_{\infty,+}((0,\infty)\times\Omega,\mathbb{R}^2)
	\end{equation*} 
and a sequence $(\tau_m)_{m\ge 1}$, $\tau_m\to 0$, such that, for any $T>0$ and $p\in [1,\infty)$,
	\begin{equation}
		\begin{split}
			u^{\tau_m} & \longrightarrow u \;\;\text{ in }\;\; L_p((0,T)\times\Omega,\mathbb{R}^2)\,, \\
			u^{\tau_m} & \stackrel{*}{\rightharpoonup} u \;\;\text{ in }\;\; L_\infty((0,T)\times\Omega,\mathbb{R}^2)\,, \\
			u^{\tau_m} & \rightharpoonup u \;\;\text{ in }\;\; L_2((0,T),H^1(\Omega,\mathbb{R}^2))\,.
		\end{split} \label{x13}
	\end{equation}
	In addition, the compact embedding of $L_2(\Omega,\mathbb{R}^2)$ in $H^1(\Omega,\mathbb{R}^2)'$, along with \eqref{x06} with $n=2$, \eqref{x10}, and \eqref{x11}, allows us to apply once more \cite[Theorem~1]{DJ2012} to conclude that 
	\begin{equation}
		u\in C([0,\infty),H^1(\Omega,\mathbb{R}^2)')\,. \label{x14}
	\end{equation}
	
	Let us now identify the equations solved by the components $f$ and $g$ of $u$. 
	To this end, let~${\chi\in W^1_\infty([0,\infty))}$ be a compactly supported function and $\varphi\in C^1(\overline{\Omega})$. In view of  \eqref{x03a}, classical computations give 
	\begin{align*}
		& \int_0^\infty \int_\Omega \frac{\chi(t+\tau)-\chi(t)}{\tau} f^\tau(t) \varphi\ \mathrm{d}x\mathrm{d}t + \left( \frac{1}{\tau} \int_0^\tau \chi(t)\ \mathrm{d}t \right) \int_\Omega f^{in}\varphi\ \mathrm{d}x 	\\
		& \qquad = \int_0^\infty \int_\Omega \chi(t) f^\tau(t) \nabla[ af^\tau(t) + b g^\tau(t)] \cdot\nabla\varphi\ \mathrm{d}x\mathrm{d}t\,.
	\end{align*}
	Taking $\tau=\tau_m$ in the above identity, it readily follows from \eqref{x13} and the regularity of $\chi$ and $\varphi$ that we may pass to the limit as $m\to\infty$ and conclude that
	\begin{equation}
		\begin{split}
		& \int_0^\infty \int_\Omega \frac{d\chi}{dt}(t) f(t,x) \varphi(x)\ \mathrm{d}x\mathrm{d}t + \chi(0) \int_\Omega f^{in}(x)\varphi(x)\ \mathrm{d}x \\
		& \qquad = \int_0^\infty \int_\Omega \chi(t) f(t,x) \nabla\left[ af+ b g\right](t,x) \cdot\nabla \varphi(x)\ \mathrm{d}x\mathrm{d}t\,.
		\end{split}
		\label{x15}
	\end{equation}
	Since $f\nabla f$ and $f\nabla g$ belong to $L_2((0,T)\times\Omega)$ for all $T>0$ by \eqref{x13}, a density argument ensures that the identity~\eqref{x15} is valid for any $\varphi\in H^1(\Omega)$. We next use the time continuity \eqref{x14} of $f$ and a classical approximation argument to show that $f$ solves \eqref{p2a}. A similar argument allows us to derive \eqref{p2b} from \eqref{x03b}.
	
	Finally, combining \eqref{x13}, \eqref{x14}, and a weak lower semicontinuity argument, we may let $m\to\infty$ in \eqref{x06}, \eqref{x07}, and \eqref{x08} with $\tau=\tau_m$ to show that $u=(f,g)$ satisfies \eqref{p3}, \eqref{p4a}, and \eqref{p5}, thereby completing the proof.
\end{proof}

%%%%%%%%%%%%%%%%
%%%%%%%%%%%%%%%%
\section*{Acknowledgments}
%%%%%%%%%%%%%%%%
%%%%%%%%%%%%%%%%

PhL gratefully acknowledges the hospitality and support of the Fakult\"at f\"ur Mathematik, Universit\"at Regensburg, where part of this work was done.

%%%%%%%%%%%%%%%%
%%%%%%%%%%%%%%%%
\appendix
\section{The polynomials $\Phi_n$, $n\ge 2$}\label{secA}
%%%%%%%%%%%%%%%%
%%%%%%%%%%%%%%%%

Let $n\ge 2$. According to the discussion in the introduction, we look for an homogeneous polynomial $\Phi_n$ of degree $n$ such that:
\begin{itemize}
	\item [(P1)] $\Phi_n$ is convex on $[0,\infty)^2$;
	\item [(P2)] the matrix $S_n(X) := D^2\Phi_n(X) M(X)$ is symmetric and positive semidefinite for~${X\in [0,\infty)^2}$.
\end{itemize}
We recall that the mobility matrix $M(X)$ is given by
\begin{equation*}
	M(X) = (m_{jk}(X))_{1\le j,k\le 2} := 
	\begin{pmatrix}
		a X_1 & b X_1 \\
		c X_2 & d X_2
	\end{pmatrix}\,, \qquad X\in\mathbb{R}^2\,,
\end{equation*}
see~\eqref{x1}. Specifically, we set
\begin{equation}
\Phi_n(X) := \sum_{j=0}^n a_{j,n} X_1^j X_2^{n-j}\,, \qquad X=(X_1,X_2)\in \mathbb{R}^2\,, \label{Polen}
\end{equation}
with  $a_{j,n}$, $0\leq j\leq n$, to be determined in order for properties~(P1)-(P2) to be satisfied. We recall that the parameters $(a,\, b,\, c,\, d)$ are assumed to satisfy~\eqref{condabcd}.

%%%%%%%%%%%%%%%%
\begin{lemma}\label{lem.p1}
Set $a_{0,n}:=1$ and
\begin{equation}\label{ajn}
	a_{j,n}:= \prod_{k=0}^{j-1}\frac{(n-k)[ak+c(n-k-1)]}{(k+1)[bk+d(n-k-1)]} = \binom{n}{j} \prod_{k=0}^{j-1}\frac{ak+c(n-k-1)}{bk+d(n-k-1)}\,,\qquad 1\leq j\leq n\,.
\end{equation}
Then $a_{j,n}>0$ for $0\le j \le n$ and $S_n(X) = D^2\Phi_n(X) M(X)\in {\mathbf{Sym}_2(\mathbb{R})}$ for all $X\in\mathbb{R}^2$.
\end{lemma}
%%%%%%%%%%%%%%%%

\begin{proof}
	Given $X\in\mathbb{R}^2$,  we compute 
	\begin{align*} 
		\partial_1^2 \Phi_n(X) & = \sum_{j=1}^{n-1} j(j+1) a_{j+1,n} X_1^{j-1} X_2^{n-j-1} = \sum_{j=0}^{n-2} (j+1)(j+2) a_{j+2,n} X_1^j X_2^{n-j-2} \,, \\
		\partial_1 \partial_2 \Phi_n(X) & = \sum_{j=1}^{n-1} j(n-j) a_{j,n} X_1^{j-1} X_2^{n-j-1} = \sum_{j=0}^{n-2} (j+1)(n-j-1) a_{j+1,n} X_1^j X_2^{n-j-2} \,, \\
		\partial_2^2 \Phi_n(X) &  = \sum_{j=0}^{n-2} (n-j)(n-j-1) a_{j,n} X_1^j X_2^{n-j-2} \,. 
	\end{align*}
	It then follows that
	\begin{align*}
		[S_n(X)]_{11} & = a X_1 \partial_1^2 \Phi_n(X) + c X_2 \partial_1 \partial_2 \Phi_n(X)\\ 
		& = a \sum_{j=1}^{n-1} j(j+1) a_{j+1,n} X_1^j X_2^{n-j-1} + c \sum_{j=0}^{n-2} (j+1)(n-j-1) a_{j+1,n} X_1^j X_2^{n-j-1}\,, \\
		[S_n(X)]_{12} & = bX_1 \partial_1^2\Phi_n(X) + d X_2 \partial_1 \partial_2 \Phi_n(X) \\
		& = b \sum_{j=1}^{n-1} j(j+1) a_{j+1,n} X_1^j X_2^{n-j-1} + d \sum_{j=0}^{n-2} (j+1)(n-j-1) a_{j+1,n} X_1^j X_2^{n-j-1}\,\\
		& = bn(n-1)a_{n,n}X_1^{n-1} +\sum_{j=1}^{n-2} (j+1)[bj+d(n-j-1)] a_{j+1,n} X_1^j X_2^{n-j-1} \\
		& \qquad +   d(n-1) a_{1,n} X_2^{n-1}\,,\\
		[S_n(X)]_{21} & = aX_1 \partial_1 \partial_2 \Phi_n(X) + c X_2 \partial_2^2 \Phi_n(X)\\
		& = a \sum_{j=1}^{n-1} j(n-j) a_{j,n} X_1^{j} X_2^{n-j-1}  +c \sum_{j=0}^{n-2} (n-j)(n-j-1) a_{j,n} X_1^j X_2^{n-j-1}\,\\ 
		& = a (n-1)a_{n-1,n}X_1^{n-1}+\sum_{j=1}^{n-2} (n-j)[aj+c(n-j-1)] a_{j,n} X_1^{j} X_2^{n-j-1} \\
		& \qquad  + c  n (n-1)a_{0,n}  X_2^{n-1}\,,\\ 
		[S_n(X)]_{22} & = b X_1 \partial_1 \partial_2 \Phi_n(X) + d X_2 \partial_2^2 \Phi_n(X)\\
		& = b \sum_{j=1}^{n-1} j(n-j) a_{j,n} X_1^{j} X_2^{n-j-1} + d \sum_{j=0}^{n-2} (n-j)(n-j-1) a_{j,n} X_1^j X_2^{n-j-1}\,.
	\end{align*}
	Hence, $S_n(X)$ is symmetric provided that
	\begin{equation*}
		(j+1)[bj+d(n-j-1)]a_{j+1,n}= (n-j)[aj+c(n-j-1)] a_{j,n},\quad 0\leq j\leq n-1\,,
	\end{equation*}
or, equivalently,
\begin{equation}\label{recursive}
	a_{j+1,n} = \frac{(n-j)[aj+c(n-j-1)]}{(j+1)[bj+d(n-j-1)]} a_{j,n}\,,\quad 0\leq j\leq n-1\,.
\end{equation}
Since $a_{0,n}=1$, the closed form formula~\eqref{ajn} readily follows from~\eqref{recursive} and we deduce from~\eqref{ajn} and the positivity of $(a,\,b,\,c,\, d)$ that $a_{j,n}>0$ for all $0\leq j\leq n$.
\end{proof}

We next show that ${D^2\Phi_n(X)}$ is positive definite for $X\in[0,\infty)^2\setminus\{(0,0)\}$. This property implies in particular that ${D^2\Phi_n(X)}$ is positive semidefinite for $X\in[0,\infty)^2$.

%%%%%%%%%%%%%%%%
\begin{lemma}\label{lem.p2}
	Let $\Phi_n$ be the polynomial defined by~\eqref{Polen} and~\eqref{ajn}. Then $D^2\Phi_n(X)\in \mathbf{SPD}_2(\mathbb{R})$ for~${X\in [0,\infty)^2\setminus\{(0,0)\}}$.
\end{lemma}
%%%%%%%%%%%%%%%%

\begin{proof}
Given $X\in [0,\infty)^2$, it follows from the positivity of the coefficients $a_{j,n}$, $0\leq j\leq n$, of $\Phi_n$ that
\begin{equation*}
	\mathrm{tr}(D^2\Phi_n(X)) := \partial_1^2 \Phi_n(X) + \partial_2^2 \Phi_n(X) \ge 0\,, \qquad X\in [0,\infty)^2\,.  
\end{equation*}
It remains to show that the determinant $\det(D^2\Phi_n(X))$ is also non-negative.
To this end we compute
\begin{align}
	\det(D^2\Phi_n(X)) & =  \partial_1^2 \Phi_n(X) \partial_2^2 \Phi_n(X) - [\partial_1 \partial_2 \Phi_n(X)]^2 \nonumber \\
	& = \sum_{j=0}^{n-2} \sum_{k=0}^{n-2} (j+1)(n-k-1) A_{j,k} X_1^{j+k} X_2^{2n-j-k-4}\,, \label{lf07}
\end{align}
where
\begin{equation*}
	A_{j,k} := (j+2)(n-k) a_{j+2,n} a_{k,n} - (n-j-1)(k+1) a_{j+1,n} a_{k+1,n}\,, \qquad 0 \le j,k \le n-2\,.
\end{equation*}
Using \eqref{recursive}, we express $a_{j+2,n}$ and $a_{k+1,n}$ in terms of $a_{j+1,n}$ and $a_{k,n}$, respectively, to arrive at the following formula
\begin{align}
A_{j,k}&= (n-k)(n-j-1) \Big[ \frac{a(j+1)+c(n-j-2)}{b(j+1)+d(n-j-2)}  - \frac{ak+c(n-k-1)}{bk+d(n-k-1)}\Big]  a_{j+1,n} a_{k,n} \nonumber\\[1ex]
&= (ad-bc)(n-k)(n-j-1) \frac{(j+1)(n-k-1)-k(n-j-2)}{[b(j+1)+d(n-j-2)] [bk+d(n-k-1)]}  a_{j+1,n} a_{k,n}\nonumber \\[1ex]
&= (ad-bc)\frac{(n-1)(n-k)(n-j-1)(j+1-k)}{\alpha_{j+1,n}\alpha_{k,n}}   a_{j+1,n} a_{k,n}\,, \label{Ajk}
\end{align}
where $\alpha_{k,n}$ denotes the positive number
\begin{equation*}
\alpha_{k,n}:=bk+d(n-k-1)\,,\quad 0\leq k\leq n-1\,.
\end{equation*}
 In particular, 
\begin{equation}
	A_{k-1,j+1} = - A_{j,k}\,, \qquad 0\le j\le n-3\,, \ 1\le k\le n-2\,. \label{lf09}
\end{equation} 
It then follows from \eqref{lf07}   that
\begin{align*}
	2 \det(D^2\Phi_n(X))	& = \sum_{j=0}^{n-2} \sum_{k=0}^{n-2} (j+1)(n-k-1) A_{j,k} X_1^{j+k} X_2^{2n-j-k-4} \\
	& \qquad + \sum_{l=1}^{n-1} \sum_{i=-1}^{n-3} l(n-i-2) A_{l-1,i+1} X_1^{i+l} X_2^{2n-i-l-4} \\
	& = \sum_{j=0}^{n-2} \sum_{k=0}^{n-2} (j+1)(n-k-1) A_{j,k} X_1^{j+k} X_2^{2n-j-k-4} \\
	& \qquad + \sum_{j=-1}^{n-3} \sum_{k=1}^{n-1} k(n-j-2) A_{k-1,j+1} X_1^{j+k} X_2^{2n-j-k-4}\\
	& = \sum_{j=0}^{n-3} \sum_{k=1}^{n-2} (j+1)(n-k-1) A_{j,k} X_1^{j+k} X_2^{2n-j-k-4} \\
	& \qquad + \sum_{k=0}^{n-2} (n-1)(n-k-1) A_{n-2,k} X_1^{n-2+k} X_2^{n-k-2} \\
	& \qquad + \sum_{j=0}^{n-3} (j+1)(n-1) A_{j,0} X_1^{j} X_2^{2n-j-4} \\	
	& \qquad  + \sum_{j=0}^{n-3} \sum_{k=1}^{n-2} k(n-j-2) A_{k-1,j+1} X_1^{j+k} X_2^{2n-j-k-4} \\
	& \qquad + \sum_{k=1}^{n-1}  k(n-1) A_{k-1,0} X_1^{k-1} X_2^{2n-k-3} \\
	& \qquad + \sum_{j=0}^{n-3} (n-1)(n-j-2) A_{n-2,j+1} X_1^{j+n-1} X_2^{n-j-3} \,.
\end{align*}
According to~\eqref{condabcd} and~\eqref{Ajk},
\begin{align*}
	A_{l,0} & = (ad-bc) \frac{n(n-1)(n-1-l)(l+1)}{\alpha_{0,n}\alpha_{l+1,n}}>0\,, \qquad 0\le l\le n-2\,,\\
	A_{n-2,l} & = (ad-bc) \frac{(n-1)(n-l)(n-1-l)}{\alpha_{n-1,n}\alpha_{l,n}}>0\,, \qquad 0\le l\le n-2\,.
\end{align*}
In particular, all the terms in the above identity involving a single sum are non-negative. Therefore, using the symmetry property~\eqref{lf09} and retaining in the last two sums only the terms corresponding to $k=1$ and $j=n-3$, respectively, we get
\begin{align*}
	2 \det(D^2\Phi_n(X))	& \ge	 \sum_{j=0}^{n-3} \sum_{k=1}^{n-2} \left[ (j+1)(n-k-1) - k(n-j-2) \right] A_{j,k} X_1^{j+k} X_2^{2n-j-k-4} \\
	& \qquad + (n-1) A_{n-2,n-2} X_1^{2n-4} + (n-1) A_{0,0} X_2^{2n-4} \\
	& = \sum_{j=0}^{n-3} \sum_{k=1}^{n-2} (n-1)(j+1-k) A_{j,k} X_1^{j+k} X_2^{2n-j-k-4} \\
	& \qquad + (n-1) A_{n-2,n-2} X_1^{2n-4} + (n-1) A_{0,0} X_2^{2n-4}\,.
\end{align*}
Observing that
\begin{equation*}
	(n-1)(j+1-k) A_{j,k} = (ad-bc) \frac{(n-1)^2(n-k)(n-j-1)(j+1-k)^2}{\alpha_{j+1,n}\alpha_{k,n}} a_{j+1,n} a_{k,n}  \ge 0
\end{equation*}
for $0\le j, k\le n-2$, we conclude that
\begin{equation}
	2\det(D^2\Phi_n(X)) \ge (n-1) A_{n-2,n-2} X_1^{2n-4} + (n-1) A_{0,0} X_2^{2n-4}\,, \qquad X\in [0,\infty)^2\,. \label{lf10}
\end{equation} 
Since $A_{0,0}>0$ and $A_{n-2,n-2}>0$, we have thus established that, for each $X\in [0,\infty)^2\setminus\{(0,0)\}$, the symmetric matrix $D^2\Phi_n(X)$ has non-negative trace and positive determinant, so that it is positive definite. 
\end{proof}

We next turn to the positive definiteness of $S_n = D^2\Phi_n M$.

%%%%%%%%%%%%%%%%
\begin{lemma}\label{lem.p3}
	Let $\Phi_n$ be defined by~\eqref{Polen} and~\eqref{ajn}. Then $S_n(X)=D^2\Phi_n(X)M(X)\in \mathbf{SPD}_2(\mathbb{R})$ for $X\in (0,\infty)^2$.
\end{lemma}
%%%%%%%%%%%%%%%%

\begin{proof}
Let $X\in (0,\infty)^2$. On the one hand, by~\eqref{condabcd}, \eqref{lf10}, and the positivity of $A_{0,0}$ and $A_{n-2,n-2}$,
\begin{align*}
	2\det(S_n(X)) & = 2(ad-bc)X_1 X_2 \det(D^2\Phi_n(X)) \\
	& \geq(ad-bc)X_1X_2(n-1) \Big[ A_{n-2,n-2} X_1^{2n-4} + (n-1) A_{0,0} X_2^{2n-4} \Big]>0\,.
\end{align*}
On the other hand, the positivity of $a_{j,n}$ for $0\le j \le n$ and \eqref{condabcd} imply that
\begin{equation*}
	\mathrm{tr}(S_n(X)) = [S_n(X)]_{11} + [S_n(X)]_{22}> 0\,.
\end{equation*}
Consequently, $S_n(X)$ has positive trace and positive determinant, and is thus positive definite as claimed.
\end{proof}

We end up this section with useful upper and lower bounds for $\Phi_n$.

%%%%%%%%%%%%%%%%
\begin{lemma}\label{lelfb}
	Let $\Phi_n$ be defined by~\eqref{Polen} and~\eqref{ajn}. Then
	\begin{equation}\label{LUB}
		\frac{( c X_1 + d X_2)^n}{d^n} \le \Phi_n(X) \le \frac{( a X_1 + b X_2)^n}{b^n}\,, \qquad X\in [0,\infty)^2\,.
	\end{equation}
\end{lemma}
%%%%%%%%%%%%%%%%

\begin{proof}
Since the function
\begin{equation*}
	\chi (z) := \frac{ (a-c) z+c}{(b-d) z+d}\,, \qquad z\in [0,1]\,,
\end{equation*}
is increasing and positive, we deduce from \eqref{ajn} that, for $1\le j \le n$, 
\begin{equation*}
	a_{j,n}=  \binom{n}{j} \prod_{k=0}^{j-1}\chi\left( \frac{k}{n-1} \right)\leq \binom{n}{j} [\chi(1)]^j=\binom{n}{j}  \left( \frac{a}{b} \right)^j
\end{equation*} 
and
\begin{equation*}
	a_{j,n}=  \binom{n}{j} \prod_{k=0}^{j-1}\chi\left( \frac{k}{n-1} \right)\geq \binom{n}{j} [\chi(0)]^j=\binom{n}{j}  \left( \frac{c}{d} \right)^j\,.
\end{equation*} 
The  upper and lower bounds in \eqref{LUB} are direct consequences of the above inequalities.
\end{proof}

\bibliographystyle{siam}
\bibliography{GTFMP}
\end{document}